\documentclass[11pt,reqno]{amsart}
\usepackage[letterpaper,margin=1in,footskip=0.25in]{geometry}
\usepackage{amssymb}
\usepackage{mathrsfs}
\usepackage{microtype}
\usepackage{mathtools}
\usepackage{enumitem}
\usepackage{longtable,booktabs}
\usepackage[referable]{threeparttablex}
\usepackage[tableposition=top]{caption}
\usepackage[figuresright]{rotating}
\usepackage{pdflscape} 
\usepackage[noadjust]{cite}
\renewcommand{\citepunct}{;\penalty\citemidpenalty\ }

\PassOptionsToPackage{dvipsnames}{xcolor}
\usepackage{tikz-cd}

\PassOptionsToPackage{pdfusetitle,colorlinks,pagebackref}{hyperref}
\usepackage{bookmark}
\hypersetup{citecolor=OliveGreen,linkcolor=Mahogany,urlcolor=Plum}
\usepackage[hyphenbreaks]{breakurl}

\newtheorem{theorem}{Theorem}[section]
\newtheorem{corollary}[theorem]{Corollary}
\newtheorem{lemma}[theorem]{Lemma}
\newtheorem{problem}[theorem]{Problem}
\newtheorem{proposition}[theorem]{Proposition}
\newtheorem{question}[theorem]{Question}

\newtheorem{innercustomthm}{Theorem}
\newenvironment{customthm}[1]
  {\renewcommand\theinnercustomthm{#1}\innercustomthm}
  {\endinnercustomthm}

\theoremstyle{definition}
\newtheorem{definition}[theorem]{Definition}
\newtheorem{conditions}[theorem]{Conditions}

\theoremstyle{remark}
\newtheorem{remark}[theorem]{Remark}

\newtheoremstyle{cited}{.5\baselineskip\@plus.2\baselineskip\@minus.2\baselineskip}{.5\baselineskip\@plus.2\baselineskip\@minus.2\baselineskip}{\itshape}{}{\bfseries}{\bfseries .}{5pt plus 1pt minus 1pt}{\thmname{#1}\thmnumber{ #2}\thmnote{ \normalfont#3}}
\theoremstyle{cited}
\newtheorem{citedthm}[theorem]{Theorem}
\newtheorem{citedlem}[theorem]{Lemma}
\newtheorem{citedprob}[theorem]{Problem}

\newtheoremstyle{citeddef}{.5\baselineskip\@plus.2\baselineskip\@minus.2\baselineskip}{.5\baselineskip\@plus.2\baselineskip\@minus.2\baselineskip}{}{}{\bfseries}{\bfseries .}{5pt plus 1pt minus 1pt}{\thmname{#1}\thmnumber{ #2}\thmnote{ \normalfont#3}}
\theoremstyle{citeddef}
\newtheorem{citeddef}[theorem]{Definition}

\newtheoremstyle{step}{.25\baselineskip\@plus.1\baselineskip\@minus.1\baselineskip}{.25\baselineskip\@plus.1\baselineskip\@minus.1\baselineskip}{\itshape}{}{\bfseries}{\bfseries .}{5pt plus 1pt minus 1pt}{\thmname{#1}\thmnumber{ #2}\thmnote{ \normalfont(#3)}}
\theoremstyle{step}
\newtheorem{step}{Step}[theorem]

\DeclareMathOperator{\Ass}{Ass}
\DeclareMathOperator{\Frac}{Frac}
\DeclareMathOperator{\Spec}{Spec}
\DeclareMathOperator{\Supp}{Supp}
\DeclareMathOperator{\cid}{cid}
\DeclareMathOperator{\cmd}{cmd}
\DeclareMathOperator{\depth}{depth}
\DeclareMathOperator{\height}{ht}

\newcommand{\CC}{\mathbf{C}}
\newcommand{\QQ}{\mathbf{Q}}
\newcommand{\RR}{\mathbf{R}}
\newcommand{\cO}{\mathcal{O}}

\newcommand{\fm}{\mathfrak{m}}
\newcommand{\fn}{\mathfrak{n}}
\newcommand{\fp}{\mathfrak{p}}
\newcommand{\fq}{\mathfrak{q}}
\newcommand{\sC}{\mathscr{C}}
\newcommand{\alt}{\mathrm{alt}}
\newcommand{\id}{\mathrm{id}}
\newcommand{\red}{\mathrm{red}}
\newcommand{\sh}{\mathrm{sh}}

\newcommand{\bP}{\mathbf{P}}
\newcommand{\bR}{\mathbf{R}}

\mathchardef\mhyphen="2D

\providecommand\given{}
\newcommand\SetSymbol[1][]{\nonscript\:#1\vert\allowbreak\nonscript\:\mathopen{}}
\DeclarePairedDelimiterX\Set[1]\{\}{\renewcommand\given{\SetSymbol[\delimsize]}#1}

\begin{document}
\title[A uniform treatment of Grothendieck's localization problem]{A uniform treatment of\\Grothendieck's localization problem}
\author{Takumi Murayama}
\address{Department of Mathematics\\Princeton University\\Princeton, NJ
08544-1000\\USA}
\email{\href{mailto:takumim@math.princeton.edu}{takumim@math.princeton.edu}}
\urladdr{\url{https://web.math.princeton.edu/~takumim/}}

\thanks{This material is based upon work supported by the National Science
Foundation under Grant Nos.\ DMS-1701622 and DMS-1902616}

\makeatletter
  \hypersetup{
    pdfsubject=\@subjclass,pdfkeywords=\@keywords
  }
\makeatother

\begin{abstract}
  Let $f\colon Y \to X$ be a proper flat morphism of locally noetherian schemes.
  Then, the locus in $X$ over which $f$ is smooth is stable under generization.
  We prove that under suitable assumptions on the formal fibers
  of $X$, the same property holds for other local
  properties of morphisms, even if $f$ is only closed and flat.
  Our proof of this statement reduces to a purely local question known as
  Grothendieck's localization problem.
  To answer Grothendieck's problem, we provide a general framework that gives a
  uniform treatment of previously known cases of this problem, and also solves
  this problem in new cases, namely for weak normality, seminormality,
  $F$-rationality, and the property ``Cohen--Macaulay and
  $F$-injective.''
  For the weak normality statement, we prove that weak normality always lifts
  from Cartier divisors.
  We also solve Grothendieck's localization problem for terminal, canonical,
  and rational singularities in equal characteristic zero.
\end{abstract}

\maketitle
\section{Introduction}\label{sect:intro}
Let $f\colon Y \to X$ be a proper flat morphism of locally noetherian schemes.
By \cite[Thm.\ 12.2.4$(iii)$]{EGAIV3}, the locus of points $x \in X$ such that
$f^{-1}(x)$ is smooth over $\kappa(x)$ is open, and in particular, is stable
under generization.
In \cite[(12.0.2)]{EGAIV3}, Grothendieck and Dieudonn\'e asked whether similar
statements hold for other local properties of morphisms, in the
following sense:
\begin{question}\label{ques:closedpointssuffice}
  Let $\bR$ be a property of noetherian local rings,
  and consider a proper flat morphism $f\colon Y \to X$ of locally noetherian
  schemes.
  Then, is the locus
  \[
    U_\bR(f) \coloneqq \Set[\big]{x \in X \given f^{-1}(x)\ \text{is
    geometrically}\ \bR\ \text{over}\ \kappa(x)} \subseteq X
  \]
  stable under generization?
\end{question}
\par Question \ref{ques:closedpointssuffice} was answered for many properties
$\bR$ in \cite[\S12]{EGAIV3}, and is a global version of Problem
\ref{prob:grothlocprob} below, which is known as Grothendieck's
localization problem (see, e.g., \cite{AF94}).
Our goal is to provide a general framework with
which to answer Question \ref{ques:closedpointssuffice}, assuming that $\bR$ is
well-behaved in the sense that it satisfies
the following four permanence conditions: 
\begin{enumerate}
  \item[{$(\hyperref[cond:ascentphom]{\textup{R}'_{\textup{I}}})$}]
    (Ascent via geometrically $\bR$ homomorphisms)
    For every flat local homomorphism $A \to B$ of
    noetherian local rings with geometrically $\bR$ fibers, if $A$ satisfies
    $\bR$, then $B$ satisfies $\bR$.
  \item[{$(\hyperref[cond:descent]{\textup{R}_{\textup{II}}})$}]
    (Descent)
    For every flat local homomorphism $A \to B$ of
    noetherian local rings, if $B$ satisfies $\bR$,
    then $A$ satisfies $\bR$.
  \item[{$(\hyperref[cond:deforms]{\textup{R}_{\textup{IV}}})$}]
    (Lifting from Cartier divisors) For every noetherian local ring $A$
    and for every nonzerodivisor $t$ in its maximal ideal, if $A/tA$
    satisfies $\bR$, then $A$ satisfies $\bR$.
  \item[{$(\hyperref[cond:generizes]{\textup{R}_{\textup{V}}})$}]
    (Localization) If a noetherian
    local ring $A$ satisfies $\bR$,
    then $A_\fp$ satisfies $\bR$ for every prime ideal $\fp \subseteq
    A$.
\end{enumerate}
These conditions on $\bR$ are studied in \cite[\S7]{EGAIV2} (see also
Conditions \ref{cond:listof}), and are satisfied by many common
properties $\bR$ (see Table \ref{table:condslist}).
The notation $(\hyperref[cond:ascentphom]{\textup{R}'_{\textup{I}}})$ is used
instead of $(\hyperref[cond:ascentreghom]{\textup{R}_{\textup{I}}})$ because the
latter condition in \cite[(7.3.10)]{EGAIV2} asserts that $\bR$ ascends via
geometrically \emph{regular} homomorphisms.
The condition $(\hyperref[cond:deforms]{\textup{R}_{\textup{IV}}})$ is
called ``deformation'' in commutative algebra, is related to inversion of
adjunction-type results in birational geometry, and can also be thought of as an
inverse to local Bertini-type theorems.
The terminology ``lifts from Cartier divisors'' was suggested to us by J\'anos
Koll\'ar.\medskip
\par Our main result says that under an additional assumption on the formal
fibers of the local rings of $X$, a more general version of Question
\ref{ques:closedpointssuffice} holds.
\begin{customthm}{\ref{thm:grothlocprobglobalintro}}
  Let $\bR$ be a property of noetherian local rings satisfying
  $(\hyperref[cond:ascentphom]{\textup{R}'_{\textup{I}}})$,
  $(\hyperref[cond:descent]{\textup{R}_{\textup{II}}})$,
  $(\hyperref[cond:deforms]{\textup{R}_{\textup{IV}}})$, and
  $(\hyperref[cond:generizes]{\textup{R}_{\textup{V}}})$, such that regular
  local rings satisfy $\bR$.
  Consider a flat morphism $f\colon Y \to X$ of locally noetherian
  schemes.
  \begin{enumerate}[label=$(\roman*)$,ref=\roman*]
    \item Suppose that $f$ maps closed to closed points, and that the local
      rings of $X$ at closed points have geometrically $\bR$ formal fibers.
      If every closed fiber of $f$ is geometrically $\bR$, then all fibers of
      $f$ are geometrically $\bR$.
    \item Suppose that $f$ is closed, and that the local rings of $X$ have
      geometrically $\bR$ formal fibers.
      Then, the locus
      \[
        U_\bR(f) \coloneqq \Set[\big]{x \in X \given f^{-1}(x)\ \text{is
        geometrically}\ \bR\ \text{over}\ \kappa(x)} \subseteq X
      \]
      is stable under generization.
  \end{enumerate}
\end{customthm}
\par In the statement above, a locally noetherian scheme $X$ over a field $k$ is
\textsl{geometrically $\bR$} over $k$ if for all finite field extensions $k
\subseteq k'$, every local ring of $X \otimes_k k'$ satisfies $\bR$.
We consider the fiber $f^{-1}(x)$ of a morphism $f$ as a scheme over the residue
field $\kappa(x)$ at $x \in X$.
We will use similar terminology for noetherian algebras over a field $k$ and
homomorphisms of noetherian rings.
A semi-local noetherian ring $A$ has \textsl{geometrically $\bR$ formal fibers}
if the $\fm$-adic completion homomorphism $A \to \widehat{A}$ has geometrically
$\bR$ fibers, where $\fm$ is the product of the maximal ideals in $A$.
\par Theorem
\ref{thm:grothlocprobglobalintro}$(\ref{thm:grothlocprobglobalintrogenerizes})$
answers Question \ref{ques:closedpointssuffice} since proper morphisms are
closed.
Combined with the constructibility results proved in
\cite[\S9]{EGAIV3} and \cite[\S7]{BF93}, Theorem
\ref{thm:grothlocprobglobalintro}$(\ref{thm:grothlocprobglobalintrogenerizes})$
implies that many properties of fibers are open on the target for closed flat
morphisms of finite type between locally noetherian schemes with nice formal
fibers.
In Theorem
\ref{thm:grothlocprobglobalintro}$(\ref{thm:grothlocprobglobalintroclosedfib})$,
the condition that a morphism maps closed points to
closed points is much weaker than properness or even closedness, since it is
satisfied by all morphisms locally of finite type between algebraic varieties,
or more generally between Jacobson schemes \cite[Cor.\ 6.4.7 and Prop.\
6.5.2]{EGAInew}.
\par Theorem \ref{thm:grothlocprobglobalintro} does not need to
assume that $f$ is proper or even of finite type, and
therefore answers a question of Shimomoto \cite[p.\ 1058]{Shi17}.
Shimomoto proved versions of
$(\ref{thm:grothlocprobglobalintroclosedfib})$ and
$(\ref{thm:grothlocprobglobalintrogenerizes})$
for morphisms of finite type between excellent
noetherian schemes \cite[Main Thm.\ 1 and Cor.\ 3.8]{Shi17}, and asked whether
similar results hold without finite type hypotheses.\medskip
\par The main ingredient in the proof of Theorem
\ref{thm:grothlocprobglobalintro} is a purely commutative-algebraic statement,
which is of independent interest.
In \cite{EGAIV2}, Grothendieck and Dieudonn\'e asked whether the following local
version of Question \ref{ques:closedpointssuffice} holds.
\begin{problem}[Grothendieck's localization problem; see
  {\cite[Rem.\ 7.5.4$(i)$]{EGAIV2}}]\label{prob:grothlocprob}
  Let $\bR$ be a property of noetherian local rings, and
  consider a flat local homomorphism $\varphi\colon A \to B$
  of noetherian local rings.
  If $A$ has geometrically $\bR$ formal fibers and the closed fiber of $\varphi$
  is geometrically $\bR$, then are all fibers of $\varphi$ geometrically $\bR$?
\end{problem}
\par In other words, Problem \ref{prob:grothlocprob} asks whether the property
of having geometrically $\bR$ fibers localizes for flat local homomorphisms of
noetherian local rings.
We call Problem \ref{prob:grothlocprob} ``Grothendieck's localization problem''
following Avramov and Foxby \cite{AF94}, who proved many cases of
this problem; see Table \ref{table:resultslist}.
We resolve Grothendieck's localization problem
\ref{prob:grothlocprob} for well-behaved properties $\bR$, extending
\cite[Prop.\ 7.9.8]{EGAIV2} and \cite[Thm.\ 2.1]{Mar84} to nonzero residue
characteristic.
\begin{customthm}{\ref{thm:grothlocprob}}
  Grothendieck's localization problem \ref{prob:grothlocprob} holds for
  properties $\bR$ of noetherian local rings satisfying
  $(\hyperref[cond:ascentphom]{\textup{R}'_{\textup{I}}})$,
  $(\hyperref[cond:descent]{\textup{R}_{\textup{II}}})$,
  $(\hyperref[cond:deforms]{\textup{R}_{\textup{IV}}})$, and
  $(\hyperref[cond:generizes]{\textup{R}_{\textup{V}}})$, such that regular
  local rings satisfy $\bR$.
\end{customthm}
\par The proof of Theorems \ref{thm:grothlocprobglobalintro} and
\ref{thm:grothlocprob} now proceeds as follows.
\begin{enumerate}[label=$(\textup{\Roman*})$,ref=\textup{\Roman*}]
  \item\label{strategy:localq}
    We reduce Theorem \ref{thm:grothlocprobglobalintro} to Theorem
    \ref{thm:grothlocprob} by replacing $X$ and $Y$ with $\Spec(\cO_{X,x})$ and
    $\Spec(\cO_{Y,y})$ for suitable points $x \in X$ and $y \in Y$.
    This step uses the assumption either that $f$ maps closed points to closed
    points, or that $f$ is closed.
  \item\label{strategy:complete}
    We reduce to the case when $A$ is quasi-excellent by replacing $A$ with its
    completion $\widehat{A}$.
    This step uses $(\hyperref[cond:ascentphom]{\textup{R}'_{\textup{I}}})$,
    $(\hyperref[cond:descent]{\textup{R}_{\textup{II}}})$, and the condition on
    the formal fibers of $A$.
  \item\label{strategy:alteration}
    We reduce to the case when $A$ is a regular local ring by applying
    Gabber's weak local uniformization theorem \cite[Exp.\ VII, Thm.\
    1.1]{ILO14}.
    This step uses
    $(\hyperref[cond:ascentphom]{\textup{R}'_{\textup{I}}})$,
    $(\hyperref[cond:descent]{\textup{R}_{\textup{II}}})$,
    $(\hyperref[cond:deforms]{\textup{R}_{\textup{IV}}})$,
    $(\hyperref[cond:generizes]{\textup{R}_{\textup{V}}})$, and the
    quasi-excellence of $A$ obtained in $(\ref{strategy:complete})$.
  \item\label{strategy:regular}
    Finally, we are in a situation where we can apply
    \cite[Lem.\ 7.5.1.1]{EGAIV2}, which is a statement similar
    to Problem \ref{prob:grothlocprob} when $A$ is regular.
    This step uses $(\hyperref[cond:deforms]{\textup{R}_{\textup{IV}}})$.
\end{enumerate}
\par The main innovation in our approach to Theorems
\ref{thm:grothlocprobglobalintro} and \ref{thm:grothlocprob} is
the use of Gabber's theorem in $(\ref{strategy:alteration})$.
Grothendieck and Dieudonn\'e \cite[Prop.\ 7.9.8]{EGAIV2} and
Marot \cite[Thm.\ 2.1]{Mar84}
separately proved versions of Theorem
\ref{thm:grothlocprob} under the assumption that every reduced module-finite
$A$-algebra has a resolution of singularities.
This assumption on resolutions of singularities holds when $A$ is
a quasi-excellent $\QQ$-algebra \cite[Ch.\ I, \S3, Main Thm.\ I$(n)$]{Hir64}, or
when $A$ is
quasi-excellent of dimension at most three \citeleft\citen{Lip78}\citemid
Thm.\citepunct \citen{CP}\citemid Thm.\ 1.1\citeright, but is not known to hold
in general.
\par On the other hand, Gabber's weak local uniformization theorem says that
a variant of resolutions of singularities exists for
arbitrary quasi-excellent noetherian schemes \cite[Exp.\ VII, Thm.\ 1.1]{ILO14}.
Gabber's theorem is a version of de Jong's alteration theorem
\cite[Thm.\ 4.1]{dJ96} for quasi-excellent noetherian schemes that are not
necessarily of finite type over a field or a DVR.
Both of their theorems say that variants of resolutions of
singularities exist after possibly passing to a finite extension of the
function field.
As presented by Kurano and Shimomoto \cite[Main Thm.\ 2]{KS},
Gabber previously used this theorem to prove that quasi-excellence is preserved
under ideal-adic completion.
\par We note that one can also ask about generization on the source space $Y$
in the context of Theorem \ref{thm:grothlocprobglobalintro}.
Such a statement follows from Theorem \ref{thm:grothlocprob} for arbitrary flat
morphisms $f\colon Y \to X$ of locally noetherian schemes, with appropriate
assumptions on the formal fibers of the local rings of
$X$; cf.\ \cite[Rem.\ 7.9.10$(ii)$]{EGAIV2}.\medskip
\par In the second half of this paper, we answer Question
\ref{ques:closedpointssuffice} and Problem \ref{prob:grothlocprob} in specific
cases.
Starting with Andr\'e's theorem on the localization of formal smoothness
\cite[Thm.\ on p.\ 297]{And74}, previous cases of Grothendieck's localization
problem \ref{prob:grothlocprob} were proved using a variety of methods,
including Andr\'e--Quillen homology \cite{And67,Qui70}, Grothendieck duality
\cite{Har66}, and the Cohen factorizations of Avramov--Foxby--Herzog
\cite{AFH94}.
See Table \ref{table:resultslist} for known cases of Problem
\ref{prob:grothlocprob}.
By checking the conditions
$(\hyperref[cond:ascentphom]{\textup{R}'_{\textup{I}}})$,
$(\hyperref[cond:descent]{\textup{R}_{\textup{II}}})$,
$(\hyperref[cond:deforms]{\textup{R}_{\textup{IV}}})$, and
$(\hyperref[cond:generizes]{\textup{R}_{\textup{V}}})$ (see Table
\ref{table:condslist}), we give a uniform treatment of most of the results in
Table \ref{table:resultslist} using Theorem \ref{thm:grothlocprob}.
Additionally, Theorem \ref{thm:grothlocprob} resolves Problem
\ref{prob:grothlocprob} for
weak normality,
seminormality, $F$-rationality, and the property ``Cohen--Macaulay and
$F$-injective,'' the latter three of which were previously known only under
additional finiteness assumptions
\citeleft\citen{Has01}\citemid Thm.\ 5.8 and Rem.\ 6.7\citepunct
\citen{Shi17}\citemid Cors.\ 3.4 and 3.10\citepunct
\citen{PSZ18}\citemid Thm.\ 5.13\citeright.
The result for $F$-rationality completely answers a question of Hashimoto
\cite[Rem.\ 6.7]{Has01}.
\par For weak normality, we prove that weak normality lifts from Cartier
divisors for all noetherian local rings (Proposition
\ref{prop:weaklynormaldeforms}),
extending a result of Bingener and Flenner \cite[Cor.\ 4.1]{BF93} to
non-excellent rings.
We also solve Grothendieck's localization problem for
terminal, canonical, and rational singularities in equal characteristic zero
(Corollary \ref{cor:locprobfortercanrat}) using
\cite[Prop.\ 7.9.8]{EGAIV2}.\medskip
\par As an application of Theorem \ref{thm:grothlocprob}, we consider the
following version of Grothendieck's lifting problem for semi-local rings.
\begin{problem}[Local lifting problem; cf.\ {\cite[Rem.\ 7.4.8A]{EGAIV2}}]
  \label{prob:grothliftprob}
  Let $A$ be a noetherian semi-local ring that is $I$-adically complete with
  respect to an ideal $I \subseteq A$.
  If $A/I$ has geometrically $\bR$ formal fibers, then does $A$ have
  geometrically $\bR$ formal fibers?
\end{problem}
\par We call Problem \ref{prob:grothliftprob} the ``local lifting problem''
following Nishimura and Nishimura \cite{NN87}, in order to distinguish Problem
\ref{prob:grothliftprob} from Grothendieck's original lifting problem asked in
\cite[Rem.\ 7.4.8A]{EGAIV2}, which does not restrict to semi-local
rings.
Using Theorem \ref{thm:grothlocprob} and the axiomatic approach to Problem
\ref{prob:grothliftprob} due to Brezuleanu and Ionescu \cite[Thm.\ 2.3]{BI84}, 
we give a solution to Problem \ref{prob:grothliftprob}
under the additional assumption that $A/I$ is Nagata.
This result extends \cite[Thm.\
5.2]{Mar84} to nonzero residue characteristic.
\begin{customthm}{\ref{thm:grothlocliftprob}}
  Let $\bR$ be a property of noetherian local rings that satisfies the
  hypotheses in Theorem \ref{thm:grothlocprob}.
  Suppose, moreover, that the locus
  \[
    U_\bR\bigl(\Spec(C)\bigr) \coloneqq
    \Set[\big]{\fp \in \Spec(C) \given C_\fp\ \text{satisfies $\bR$}}
    \subseteq \Spec(C)
  \]
  is open for every noetherian complete local ring $C$.
  If $A$ is a noetherian semi-local ring that is $I$-adically complete with
  respect to an ideal $I \subseteq A$, and if $A/I$ is Nagata and 
  has geometrically $\bR$ formal fibers, then $A$ is Nagata and has
  geometrically $\bR$ formal fibers.
\end{customthm}
\par
The condition on $U_\bR(\Spec(C))$ is the condition $(\ref{cond:openness})$ in
Conditions \ref{cond:listof}.
Recall that a noetherian semi-local ring is Nagata if and only if
it has geometrically reduced formal fibers; see Remark \ref{rem:zariskinagata}.
Theorem \ref{thm:grothlocliftprob} therefore answers the local lifting problem
\ref{prob:grothliftprob}
for properties $\bR$ such that regular $\Rightarrow$ $\bR$ $\Rightarrow$
reduced, and gives a uniform treatment of most known cases of
Problem \ref{prob:grothliftprob}; see Table \ref{table:resultslist}.
Theorem \ref{thm:grothlocliftprob} also implies that for semi-local Nagata
rings, the property of having geometrically $\bR$ formal fibers is preserved
under ideal-adic completion; see Corollary \ref{cor:grothlocliftprob}.\medskip
\par We end this introduction with one question that is still open.
Nishimura showed that the non-local version of Problem
\ref{prob:grothliftprob} is false for all properties $\bR$ such that regular
$\Rightarrow$ $\bR$ $\Rightarrow$ reduced \cite[Ex.\ 5.3]{Nis81}, and
Greco showed that similar questions for the properties ``universally catenary''
and ``excellent'' are also false \cite[Prop.\ 1.1]{Gre82}.
Instead, inspired by the axiomatic approach of Valabrega
\cite[Thm.\ 3]{Val78}, Imbesi conjectured that the following formulation of
Problem \ref{prob:grothliftprob} for non-semi-local rings may hold:
\begin{citedprob}[{\cite[p.\ 54]{Imb95}}]\label{prob:grothliftimbesi}
  Let $A$ be a noetherian ring that is $I$-adically complete with respect to an
  ideal $I \subseteq A$.
  If $A/I$ satisfies $\bR\mhyphen2$ and every local ring of $A/I$ has
  geometrically $\bR$ formal fibers, then is it true that $A$ satisfies
  $\bR\mhyphen2$ and that every local ring of $A$ has geometrically $\bR$
  formal fibers?
\end{citedprob}
\par
Here, a ring $A$ \textsl{satisfies $\bR\mhyphen2$} if for every $A$-algebra $B$
of finite type, the locus $U_\bR(\Spec(B))$ is open in $\Spec(B)$ \cite[Def.\
1]{Val78}.
Nishimura's aforementioned example \cite[Ex.\ 5.3]{Nis81} shows that Problem
\ref{prob:grothliftimbesi} does not hold when $\bR =$ ``reduced.''
On the other hand, Problem \ref{prob:grothliftimbesi} for
$\bR =$ ``normal'' follows from a result proved by
Brezuleanu--Rotthaus \cite[Satz 1]{BR82} and by Chiriacescu \cite[Thm.\
1.5]{Chi82} around the same time, and a result due to
Nishimura--Nishimura \cite[Thm.\ A]{NN87}.
Moreover, Gabber proved Problem \ref{prob:grothliftimbesi} for $\bR =$
``regular'' \cite[Main Thm.\ 1]{KS}.

\subsection*{Outline}
This paper is structured as follows.
\par In the first half of the paper, we set up the general framework with which we
prove Theorems \ref{thm:grothlocprobglobalintro}, \ref{thm:grothlocprob}, and
\ref{thm:grothlocliftprob}.
To so do, we define geometrically $\bR$
morphisms and formal fibers,
and review the necessary background on (quasi-)excellent rings and on
Gabber's weak local uniformization theorem in \S\ref{sect:prelim}.
In \S\ref{sect:proofs}, we state the various conditions we put on local
properties $\bR$ of noetherian rings, and then
prove Theorem \ref{thm:grothlocprob}, first in
the quasi-excellent case (Theorem \ref{thm:grothlocprobqe}), and then for
rings with geometrically $\bR$ formal fibers (Corollary
\ref{thm:grothlocprobpring}).
We obtain Theorems \ref{thm:grothlocprobglobalintro} and
\ref{thm:grothlocliftprob} as consequences.
\par The second half of the paper consists of \S\ref{sect:specific}, where we verify
the necessary conditions to apply Theorems \ref{thm:grothlocprobglobalintro},
\ref{thm:grothlocprob}, and \ref{thm:grothlocliftprob} to
specific properties $\bR$, in particular proving that
weak normality lifts from Cartier divisors (Proposition
\ref{prop:weaklynormaldeforms}).
We then solve Grothendieck's
localization problem \ref{prob:grothlocprob} for terminal, canonical, and
rational singularities in equal characteristic zero (Corollary
\ref{cor:locprobfortercanrat}) using \cite[Prop.\ 7.9.8]{EGAIV2}.
Finally, we conclude this paper with two tables:
Table \ref{table:resultslist} contains references for special cases of
Problems \ref{prob:grothlocprob} and \ref{prob:grothliftprob}, and Table
\ref{table:condslist} contains references for the
conditions used in our theorems.

\subsection*{Notation}
All rings are commutative with identity, and all ring homomorphisms are unital.
We follow the notation in Definition \ref{def:pring}
for properties $\bR$ of noetherian local rings and their associated
properties of morphisms and of formal fibers.
We also follow the notation in Conditions \ref{cond:listof} for the
conditions on $\bR$ appearing in our results.
In addition, if $\bR$ is a property of noetherian local rings, then following
\cite[Prop.\ 7.3.12]{EGAIV2}, the \textsl{$\bR$ locus} in a locally noetherian
scheme $X$ is the locus
\[
  U_\bR(X) \coloneqq \Set[\big]{x \in X \given
  \cO_{X,x}\ \text{satisfies}\ \bR} \subseteq X
\]
at which $X$ satisfies $\bR$.

\subsection*{Acknowledgments}
I am grateful to S\'andor Kov\'acs, Linquan Ma, Yohsuke Matsuzawa,
Zsolt Patakfalvi, Karl Schwede, Kazuma Shimomoto, Austyn Simpson, and Farrah
Yhee for helpful conversations.
I would especially like to thank Rankeya Datta, Charles Godfrey, J\'anos
Koll\'ar, and Mircea Musta\c{t}\u{a} for insightful comments on previous drafts
of this paper.
This paper arose out of discussions with Rankeya about Grothendieck's
localization problem for the property ``Cohen--Macaulay and $F$-injective'' that
occurred while writing \cite{DM}, and I am grateful to Rankeya
for allowing me to write this standalone paper.
Finally, I am indebted to the anonymous referees for helpful suggestions and
corrections that improved the quality of this paper.

\section{Preliminaries}\label{sect:prelim}
\subsection{Geometrically \texorpdfstring{$\bR$}{R} morphisms and
formal fibers}\label{sect:pmorphisms}
We begin by defining geometrically $\bR$ morphisms and rings with geometrically
$\bR$ formal fibers.
\begin{citeddef}[{\cite[(7.3.1), (7.5.0), and (7.3.13)]{EGAIV2}}]
  \label{def:pring}
  Let $\bR$ be a property of noetherian local rings, and let $k$ be a field.
  A locally noetherian $k$-scheme $X$ is \textsl{geometrically $\bR$} over $k$
  if $X$ satisfies the following property:
  \begin{equation}\label{eq:whategacallsp}
    \text{For all finite field extensions $k \subseteq k'$, every local ring of
    $X \otimes_k k'$ satisfies $\bR$.}
  \end{equation}
  A noetherian $k$-algebra $A$ is \textsl{geometrically $\bR$} over $k$ if
  $\Spec(A)$ is geometrically $\bR$ over $k$.
  \par A morphism $f \colon Y \to X$ of locally noetherian schemes is
  \textsl{geometrically $\bR$} if it is flat and if the scheme
  $f^{-1}(x)$ is geometrically $\bR$ for every $x \in X$.
  We consider the fiber $f^{-1}(x)$ of a morphism $f$ as a scheme over the
  residue field $\kappa(x)$ at $x \in X$.
  A ring homomorphism $\varphi\colon A \to B$ is \textsl{geometrically $\bR$} if
  $\Spec(\varphi)$ is geometrically $\bR$.
  The \textsl{geometrically $\bR$ locus} of a flat morphism $f\colon Y \to
  X$ of locally noetherian schemes is the locus
  \[
    U_\bR(f) \coloneqq \Set[\big]{x \in X \given f^{-1}(x)\ \text{is
    geometrically}\ \bR\ \text{over}\ \kappa(x)} \subseteq X
  \]
  of points in $X$ over which $f$ is geometrically $\bR$.
  \par A noetherian semi-local ring $A$ has \textsl{geometrically $\bR$ formal
  fibers} if the $\fm$-adic completion homomorphism $A \to \widehat{A}$ is
  geometrically $\bR$, where $\fm$ is the product of the maximal ideals in $A$.
\end{citeddef}
\begin{remark}
  We note that the property in \eqref{eq:whategacallsp} is called $\bP$ in
  \cite[(7.5.0)]{EGAIV2}.
  Following this terminology, geometrically $\bR$ morphisms are called
  $\bP$-morphisms in \cite[(7.3.1)]{EGAIV2}, and semi-local rings with
  geometrically $\bR$ formal fibers are called $\bP$-rings in
  \cite[(7.3.13)]{EGAIV2}.
\end{remark}

\subsection{(Quasi-)excellent rings and schemes}
We next define (quasi-)excellent rings and schemes.
In the definition below, we recall that
a noetherian ring $A$ is a \textsl{$G$-ring} if $A_\fp$ has geometrically
regular formal fibers for every prime ideal $\fp \subseteq A$ \cite[p.\
256]{Mat89}.
\begin{citeddef}[{\cite[Def.\ 7.8.2 and (7.8.5)]{EGAIV2} (cf.\ \cite[Def.\
  on p.\ 260]{Mat89})}]
  A noetherian ring $A$ is \textsl{quasi-excellent} if $A$ is a $G$-ring and if
  $A$ is $J\mhyphen2$, i.e., if for every $A$-algebra $B$ of finite type, the
  regular locus in $\Spec(B)$ is open.
  A quasi-excellent ring $A$ is \textsl{excellent} if $A$ is universally
  catenary.
  \par A locally noetherian scheme $X$ is \textsl{quasi-excellent} (resp.\
  \textsl{excellent}) if it admits an open affine covering $X =
  \bigcup_i\Spec(A_i)$, such that every $A_i$ is quasi-excellent (resp.\
  excellent).
\end{citeddef}
\par The condition $J\mhyphen2$ is the condition $\bR\mhyphen2$ for $\bR =$
``regular'' in the sense mentioned in \S\ref{sect:intro}.
We also define the following closely related notion.
\begin{citeddef}[{\cite[Ch.\ 0, D\'ef.\ 23.1.1]{EGAIV1}}]\label{def:nagata}
  A noetherian domain $A$ is \textsl{Japanese} if, for every finite extension
  $L$ of the fraction field of $A$, the integral closure of $A$ inside
  $L$ is module-finite over $A$.
  A noetherian ring $A$ is \textsl{Nagata} or \textsl{universally Japanese} if
  every domain $B$ of finite type over $A$ is Japanese.
\end{citeddef}
\par Every quasi-excellent noetherian ring is Nagata \cite[Cor.\ 7.7.3]{EGAIV2}.
\begin{remark}\label{rem:zariskinagata}
  By theorems of Zariski and Nagata \cite[Thms.\ 7.6.4 and 7.7.2]{EGAIV2} (see
  also \cite[p.\ 264]{Mat89}), a noetherian ring $A$ is Nagata if and only if
  \begin{enumerate}[label=$(\alph*)$,ref=\alph*]
    \item\label{item:nagataredring} $A_\fp$ has geometrically reduced
      formal fibers for every prime ideal $\fp \subseteq A$; and
    \item\label{item:nagatanorlocus} For every domain
      $B$ that is module-finite over $A$, the normal locus is
      open in $\Spec(B)$.
  \end{enumerate}
  For semi-local rings $A$, $(\ref{item:nagataredring})$ implies
  $(\ref{item:nagatanorlocus})$ by \cite[Thm.\ 7.6.4 and Cor.\ 7.6.5]{EGAIV2},
  and $(\ref{item:nagataredring})$ holds if and only if $A$ has
  geometrically reduced formal fibers \cite[Prop.\ 7.3.14 and Cor.\
  7.4.5]{EGAIV2}.
  Thus, a semi-local ring is Nagata if and only if it has geometrically reduced
  formal fibers.
\end{remark}
\subsection{Gabber's weak local uniformization theorem}
We now recall Gabber's weak local uniformization theorem, which is a variant of
resolutions of singularities for arbitrary
quasi-excellent noetherian schemes.
Gabber's result is a version of de Jong's alteration theorem \cite[Thm.\
4.1]{dJ96} for quasi-excellent noetherian schemes that are not necessarily of
finite type over a field or a DVR.
\par To state Gabber's result, we first need to define maximally dominating
morphisms.
We recall that a point $x$ on a scheme $X$ is \textsl{maximal} if it is the
generic point of an irreducible component of $X$ \cite[Ch.\ 0,
(2.1.1)]{EGAInew}.
We then have the following:
\begin{citeddef}[{\cite[Exp.\ II, D\'ef.\ 1.1.2]{ILO14}}]
  A morphism $f\colon Y \to X$ of schemes is \textsl{maximally dominating} if
  every maximal point of $Y$ maps to a maximal point of $X$.
\end{citeddef}
\par We now define the alteration topology on a noetherian scheme.
\begin{citeddef}[{\cite[Exp.\ II, D\'ef.\ 1.2.2 and (2.3.1)]{ILO14}}]
  Let $X$ be a noetherian scheme.
  The category $\alt/X$ is the category whose objects are reduced schemes that
  are maximally dominating, generically finite, and of finite type over $X$, and
  whose morphisms are morphisms as schemes over $X$.
  All morphisms in $\alt/X$ are maximally dominating, generically finite, and of
  finite type \cite[Exp.\ II, Prop.\ 1.1.10]{ILO14}.
  \par The \textsl{alteration topology} on $X$ is the Grothendieck
  topology on $\alt/X$ associated to the pretopology generated by
  \begin{enumerate}[label=$(\roman*)$]
    \item \'etale coverings; and
    \item proper surjective morphisms that are maximally dominating and
      generically finite.
  \end{enumerate}
\end{citeddef}
\par We will use the following alternative characterization for coverings in
the alteration topology when $X$ is irreducible.
This result implies that coverings in the alteration topology on $\alt/X$ are
coverings in Voevodsky's $h$-topology \cite[Def.\ 3.1.2]{Voe96};
see \cite[p.\ 263]{ILO14}.
\begin{citedthm}[{\cite[Exp.\ II, Thm.\ 3.2.1]{ILO14}}]\label{thm:ksdefaltcover}
  Let $X$ be an irreducible noetherian scheme.
  Then, for every finite covering $\{Y_i \to X\}_{i=1}^m$ in the alteration
  topology on $X$, there exists a proper surjective morphism $\pi\colon V \to X$
  in $\alt/X$ such that $V$ is integral, and a Zariski open covering $V =
  \bigcup_{i=1}^m V_i$, together
  with a collection $\{h_i\colon V_i \to Y_i\}_{i=1}^m$ of morphisms
  such that the diagram
  \[
    \begin{tikzcd}
      V_i \rar[hook]\dar[swap]{h_i} & V\dar{\pi}\\
      Y_i \rar{f_i} & X
    \end{tikzcd}
  \]
  commutes for every $i \in \{1,2,\ldots,m\}$, where the morphisms $V_i
  \hookrightarrow V$ are the natural open immersions.
\end{citedthm}
\par We now state a special case of Gabber's weak local uniformization
result, which is the main technical ingredient in the proof of Theorem
\ref{thm:grothlocprob}.
\begin{theorem}[{Gabber \cite[Exp.\ VII, Thm.\ 1.1]{ILO14}}]
  \label{thm:gabberunif}
  Let $X$ be a quasi-excellent noetherian scheme.
  Then, there exists a finite covering $\{Y_i \to X\}_{i=1}^m$
  in the alteration topology on $X$, such that $Y_i$ is regular and integral
  for every $i \in \{1,2,\ldots,m\}$.
\end{theorem}

\section{Grothendieck's localization problem and the local lifting problem}
\label{sect:proofs}
\par In this section, we solve Grothendieck's localization problem
\ref{prob:grothlocprob} by proving Theorem \ref{thm:grothlocprob} in a sequence
of steps.
We first fix notation for our conditions on local properties $\bR$ in
\S\ref{sect:conds}.
In \S\ref{sect:grothlocprobqe}, we prove Theorem \ref{thm:grothlocprob}
under the additional assumption that $A$ is quasi-excellent (Theorem
\ref{thm:grothlocprobqe}), in which case Gabber's weak local uniformization
theorem \ref{thm:gabberunif} applies.
We then prove Theorem \ref{thm:grothlocprob} in \S\ref{sect:grothlocprobpring}
by taking a completion to reduce to the quasi-excellent case, using
$(\hyperref[cond:ascentphom]{\textup{R}'_{\textup{I}}})$
and the fact that $A$ has geometrically $\bR$ formal fibers.
Finally, we obtain Theorems \ref{thm:grothlocprobglobalintro} and
\ref{thm:grothlocliftprob} as consequences in \S\ref{sect:globalapps} and
\S\ref{sect:liftingproblem}, respectively.
\subsection{Conditions on \texorpdfstring{$\bR$}{R}}\label{sect:conds}
We fix the following notational conventions for permanence conditions on
the local properties $\bR$.
\begin{conditions}\label{cond:listof}
  Fix a full subcategory $\sC$ of the category of noetherian rings, and
  let $\bR$ be a property of noetherian local rings.
  We consider the following conditions on the property $\bR$:
  \begin{enumerate}[label=$(\textup{R}_{\textup{\Roman*}})$,ref=\ensuremath{\textup{R}_{\textup{\Roman*}}}]
    \setcounter{enumi}{-1}
    \item[$(\textup{R}_0)$]
      \refstepcounter{enumi}
      \makeatletter
      \def\@currentlabel{\ensuremath{\textup{R}_0}}
      \makeatother
      The property $\bR$ holds for every field $k$.\label{cond:fields}
    \item[$(\textup{R}_{\textup{I}}^\sC)$]
      \refstepcounter{enumi}
      \makeatletter
      \def\@currentlabel{\ensuremath{\textup{R}_{\textup{I}}^\sC}}
      \makeatother
      (Ascent via geometrically regular homomorphisms)
      For every geometrically regular local homomorphism $\varphi\colon A \to B$
      of noetherian local rings in $\sC$, if $A$ satisfies $\bR$, then $B$
      satisfies $\bR$.\label{cond:ascentreghom}
    \item[$(\textup{R}_{\textup{II}}^\sC)$]
      \refstepcounter{enumi}
      \makeatletter
      \def\@currentlabel{\ensuremath{\textup{R}_{\textup{II}}^\sC}}
      \makeatother
      (Descent)
      For every flat local homomorphism $\varphi\colon A \to B$ of
      noetherian local rings such that $B$ is in $\sC$, if $B$ satisfies $\bR$,
      then $A$ satisfies $\bR$.\label{cond:descent}
    \item\label{cond:openness} (Openness)
      For every noetherian complete local ring $C$, the locus
      $U_\bR(\Spec(C))$ is open.
    \item[$(\textup{R}_{\textup{IV}}^\sC)$]
      \refstepcounter{enumi}
      \makeatletter
      \def\@currentlabel{\ensuremath{\textup{R}_{\textup{IV}}^\sC}}
      \makeatother
      (Lifting from Cartier divisors) For every noetherian local ring $A$ in
      $\sC$ and for every nonzerodivisor $t$ in its maximal ideal, if $A/tA$
      satisfies $\bR$, then $A$ satisfies $\bR$.\label{cond:deforms}
    \item[$(\textup{R}_{\textup{V}}^\sC)$]
      \refstepcounter{enumi}
      \makeatletter
      \def\@currentlabel{\ensuremath{\textup{R}_{\textup{V}}^\sC}}
      \makeatother
      (Localization) If a noetherian
      local ring $A$ in $\sC$ satisfies $\bR$,
      then $A_\fp$ satisfies $\bR$ for every prime ideal $\fp \subseteq
      A$.\label{cond:generizes}
  \end{enumerate}
  and the following variant of $(\ref{cond:ascentreghom})$:
  \begin{enumerate}[label=$(\textup{R}^{\sC\prime}_{\textup{\Roman*}})$,ref=\ensuremath{\textup{R}^{\sC\prime}_{\textup{\Roman*}}}]
    \item (Ascent via geometrically $\bR$ homomorphisms)
      For every local geometrically $\bR$ homomorphism $\varphi\colon A \to B$
      of noetherian local rings in $\sC$, if $A$ satisfies $\bR$, then $B$
      satisfies $\bR$.\label{cond:ascentphom}
  \end{enumerate}
  We also consider the following conditions for $\bR$ that affect geometrically
  $\bR$ homomorphisms:
  \begin{enumerate}[label=$(\textup{P}_{\textup{\Roman*}})$,ref=\ensuremath{\textup{P}_{\textup{\Roman*}}}]
    \item[$(\textup{P}_{\textup{I}}^\sC)$]
      \refstepcounter{enumi}
      \makeatletter
      \def\@currentlabel{\ensuremath{\textup{P}_{\textup{I}}^\sC}}
      \makeatother
      If $\varphi\colon A \to B$ and $\psi\colon B \to C$ are
      a geometrically $\bR$ homomorphism and a geometrically regular
      homomorphism of noetherian rings, respectively, and $B$ and $C$ are in
      $\sC$, then $\psi \circ \varphi$ is a geometrically $\bR$
      homomorphism.\label{cond:transitivityregular}
    \item[$(\textup{P}_{\textup{II}}^\sC)$]
      \refstepcounter{enumi}
      \makeatletter
      \def\@currentlabel{\ensuremath{\textup{P}_{\textup{II}}^\sC}}
      \makeatother
      (Descent) If $\varphi\colon A \to B$ and $\psi\colon B
      \to C$ are two
      homomorphisms of noetherian rings such that $\psi \circ \varphi$ is
      geometrically $\bR$, $\psi$ is faithfully flat, and $C$ is in $\sC$,
      then $\varphi$ is geometrically $\bR$.\label{cond:descentphom}
    \item Every field $k$ is geometrically $\bR$ over
      itself.\label{cond:fieldsp}
    \item[$(\textup{P}_{\textup{IV}}^\sC)$]
      \refstepcounter{enumi}
      \makeatletter
      \def\@currentlabel{\ensuremath{\textup{P}_{\textup{IV}}^\sC}}
      \makeatother (Stability under finitely generated ground field extensions)
      If a noetherian ring $A$ in $\sC$ is geometrically $\bR$ over a field
      $k$, then for all finitely generated field extensions $k \subseteq k'$,
      the ring $A \otimes_k k'$ is geometrically $\bR$ over
      $k'$.\label{cond:fgfieldext}
  \end{enumerate}
  and the following variant of $(\ref{cond:transitivityregular})$: 
  \begin{enumerate}[label=$(\textup{P}^{\sC\prime}_{\textup{\Roman*}})$,ref=\ensuremath{\textup{P}^{\sC\prime}_{\textup{\Roman*}}}]
    \item If $\varphi\colon A \to B$ and $\psi\colon B \to C$ are
      geometrically $\bR$ homomorphisms of noetherian rings such that $B$ and
      $C$ are in $\sC$, then $\psi \circ \varphi$ is also geometrically
      $\bR$.\label{cond:transitivity}
  \end{enumerate}
  We drop $\sC$ from our notation if $\sC$ is the entire category of noetherian
  rings.
\end{conditions}
\begin{remark}
  The list in Conditions \ref{cond:listof} is a subset of that in
  \cite[Conds.\ 1.1 and 1.2]{Mar84}, although our naming convention mostly
  follows \cite{EGAIV2}.
  Specifically,
  \begin{itemize}
    \item $(\ref{cond:fields})$,
      $(\hyperref[cond:ascentreghom]{\textup{R}_{\textup{I}}})$, and
      $(\hyperref[cond:descent]{\textup{R}_{\textup{II}}})$
      appear in \cite[(7.3.10)]{EGAIV2}, although
      $(\ref{cond:fields})$ is not named;
    \item $(\ref{cond:openness})$ appears in \cite[Prop.\ 7.3.18]{EGAIV2};
    \item $(\ref{cond:deforms})$ specializes to the condition
      $(\textup{R}_{\textup{IV}})$ in \cite[Thm.\ 7.5.1]{EGAIV2} when $\sC$ is
      the category of noetherian complete local rings;
    \item $(\ref{cond:generizes})$ is unrelated to the condition
      $(\textup{R}_{\textup{V}})$ in \cite[Cor.\ 7.5.2]{EGAIV2};
    \item $(\hyperref[cond:ascentphom]{\textup{R}'_{\textup{I}}})$
      appears in \cite[Rem.\ 7.3.11]{EGAIV2};
    \item $(\hyperref[cond:transitivityregular]{\textup{P}_{\textup{I}}})$,
      $(\hyperref[cond:descentphom]{\textup{P}_{\textup{II}}})$, and
      $(\ref{cond:fieldsp})$ appear in \cite[(7.3.4)]{EGAIV2};
    \item $(\hyperref[cond:fgfieldext]{\textup{P}_{\textup{IV}}})$
      appears in \cite[(7.3.6)]{EGAIV2}; and
    \item $(\hyperref[cond:transitivity]{\textup{P}'_{\textup{I}}})$
      appears in \cite[Rem.\ 7.3.5$(iii)$]{EGAIV2}.
  \end{itemize}
  See also \cite[(7.9.7)]{EGAIV2}, \cite[p.\ 201]{Val78},
  and \cite[(2.1) and (2.4)]{BI84}.
\end{remark}
\par We will use the following relationships between different conditions on
$\bR$.
\begin{lemma}[cf.\ {\citeleft\citen{EGAIV2}\citemid (7.3.10), Rem.\ 7.3.11, and
  Lem.\ 7.3.7\citepunct \citen{DM}\citemid Prop.\ 4.10\citeright}]
  \label{lem:egariimpliespj}
  Fix a full subcategory $\sC$ of the category of noetherian rings that is
  stable under homomorphisms essentially of finite type, and
  let $\bR$ be a property of noetherian local rings.
  \begin{enumerate}[label=$(\roman*)$,ref=\roman*]
    \item If $\bR$ satisfies $(\ref{cond:fields})$, then $\bR$ satisfies
      $(\ref{cond:fieldsp})$.\label{lem:egar0impliesp3}
    \item Let $\bR'$ be another property of noetherian local rings.
      Suppose that for every local geometrically $\bR'$ homomorphism $B \to C$
      of noetherian local rings in $\sC$, if $B$ satisfies $\bR$, then $C$
      satisfies $\bR$.
      If $\varphi\colon A \to B$ and $\psi\colon B \to C$ are geometrically
      $\bR$ and geometrically $\bR'$ homomorphisms of noetherian rings,
      respectively, and $B$ and $C$ are in $\sC$, then $\psi \circ \varphi$ is
      geometrically $\bR$.
      \par In particular, if $\bR$ satisfies
      $(\ref{cond:ascentreghom})$, then
      $\bR$ satisfies $(\ref{cond:transitivityregular})$, and if $\bR$ satisfies
      $(\ref{cond:ascentphom})$, then
      $\bR$ satisfies $(\ref{cond:transitivity})$.\label{lem:egar1impliesp1}
    \item If $\bR$ satisfies $(\ref{cond:descent})$, then $\bR$ satisfies
      $(\ref{cond:descentphom})$.\label{lem:egar2impliesp2}
    \item
      If $\bR$ satisfies the special cases of
      $(\ref{cond:ascentreghom})$ and
      $(\ref{cond:descent})$ when the homomorphisms $\varphi$ are essentially of
      finite type,
      then $\bR$ satisfies $(\ref{cond:fgfieldext})$.\label{lem:egar1impliesp4}
    \item If $\bR$ satisfies $(\ref{cond:fgfieldext})$, and if
      $\varphi\colon A \to B$ is a geometrically $\bR$ homomorphism of
      noetherian rings such that $B$ is in $\sC$, then for every $A$-algebra $C$
      essentially of finite type, the base change $\varphi \otimes_A \id_C
      \colon C \to B \otimes_A C$ is geometrically
      $\bR$.\label{lem:p4implieseftbasechange}
  \end{enumerate}
\end{lemma}
\begin{proof}
  $(\ref{lem:egar0impliesp3})$ is clear from definition of being geometrically
  $\bR$.
  \par To show $(\ref{lem:egar1impliesp1})$ and
  $(\ref{lem:egar2impliesp2})$, it suffices to consider the case when $A$ is a
  field $k$
  by transitivity of fibers; see \cite[Rem.\ 
  7.3.5$(ii)$]{EGAIV2}.
  Note that for $(\ref{lem:egar1impliesp1})$ (resp.\
  $(\ref{lem:egar2impliesp2})$),
  we use the hypothesis on $\sC$ to guarantee that after this
  reduction, $B$ and $C$ (resp.\ $C$) are still in $\sC$.
  Consider the composition
  \[
    k \overset{\varphi}{\longrightarrow} B \overset{\psi}{\longrightarrow} C
  \]
  of homomorphisms of noetherian rings, and consider the base change
  \begin{equation}\label{eq:compfinext}
    k' \xrightarrow{\varphi \otimes_k \id_{k'}} B \otimes_k k'
    \xrightarrow{\psi \otimes_k \id_{k'}} C \otimes_k k'
  \end{equation}
  for a finite field extension $k \subseteq k'$.
  For $(\ref{lem:egar1impliesp1})$, we first note that since $B \to B \otimes_k
  k'$ is a module-finite homomorphism, it induces finite field extensions on
  residue fields.
  Thus, the base change $\psi \otimes_k \id_{k'}$ of $\psi$ is geometrically
  $\bR'$.
  Since the local rings of $B \otimes_k k'$ satisfy $\bR$ by assumption, we see
  that the local rings of $C \otimes_k k'$ also satisfy $\bR$.
  For $(\ref{lem:egar2impliesp2})$, we note that in \eqref{eq:compfinext}, the
  local rings of $C \otimes_k k'$ satisfy $\bR$ by assumption.
  Thus, the local rings of $B \otimes_k k'$ also satisfy $\bR$ by
  $(\ref{cond:descent})$, since $\psi \otimes_k \id_{k'}$ is faithfully flat by
  base change.
  \par For $(\ref{lem:egar1impliesp4})$, let $k \subseteq k'$ be a finitely
  generated field extension.
  By \cite[Lem.\ 4.9]{DM}, there exists a finite extension $k \subseteq
  k_1$ and a diagram
  \[
    \begin{tikzcd}[sep={2.5em,between origins}]
      & k_2 &\\
      k' \arrow[dash]{ur} & & k_1\arrow[dash]{ul}\\
      & k\arrow[dash]{ul}\arrow[dash]{ur} &
    \end{tikzcd}
  \]
  of finitely generated field extensions, where $k_1 \subseteq k_2 \coloneqq (k'
  \otimes_k k_1)_\red$ is a separable field extension.
  Since $A$ is geometrically $\bR$ over $k$, the local rings of $A \otimes_k
  k_1$ satisfy $\bR$.
  We therefore see that the local rings of $A \otimes_k k_2$ also satisfy
  $\bR$, since $k_1 \to
  k_2$ is regular and both $A \otimes_k k_1$ and $A \otimes_k k_2$ are in $\sC$.
  Finally, the local rings of $A \otimes_k k'$ satisfy $\bR$ by
  $(\ref{cond:descent})$.
  \par For $(\ref{lem:p4implieseftbasechange})$,
  we note that $\varphi \otimes_A \id_C$ is flat by base change, and hence it
  suffices to show that for every prime ideal $\fq \subseteq C$, the fiber $B
  \otimes_A C \otimes_C \kappa(\fq)$ is geometrically $\bR$.
  Letting $\fp = \fq \cap A$, the field extension $\kappa(\fp) \subseteq
  \kappa(\fq)$ is finitely generated since $A \to C$ is essentially of finite
  type \cite[Prop.\ 6.5.10]{EGAInew}.
  Since
  \[
    B \otimes_A C \otimes_C \kappa(\fq) \simeq B \otimes_A \kappa(\fq) \simeq B
    \otimes_A \kappa(\fp) \otimes_{\kappa(\fp)} \kappa(\fq),
  \]
  we see that $B \otimes_A C \otimes_C \kappa(\fq)$ is geometrically $\bR$ by
  $(\ref{cond:fgfieldext})$.
\end{proof}
\subsection{Problem \ref{prob:grothlocprob} for quasi-excellent
bases}\label{sect:grothlocprobqe}
The following result solves Grothendieck's localization problem
\ref{prob:grothlocprob} when the ring $A$ is quasi-excellent.
This step corresponds to $(\ref{strategy:alteration})$ in \S\ref{sect:intro},
and forms the technical core of the proof of Theorem \ref{thm:grothlocprob}.
\begin{theorem}[cf.\ \citeleft\citen{EGAIV2}\citemid Prop.\ 7.9.8\citepunct
  \citen{Mar84}\citemid Thm.\ 2.1\citeright]\label{thm:grothlocprobqe}
  Fix a full subcategory $\sC$ of the category of noetherian rings that is
  stable under homomorphisms essentially of finite type.
  Let $\bR$ be a property of noetherian local rings, and consider a flat local
  homomorphism $\varphi\colon (A,\fm,k) \to (B,\fn,l)$ of
  noetherian local rings.
  Assume the following:
  \begin{enumerate}[label=$(\roman*)$,ref=\roman*]
    \item\label{thm:qecond} The ring $A$ is quasi-excellent;
    \item\label{thm:qecatcond} The ring $B$ appears in $\sC$; and
    \item\label{thm:qerconds} The property $\bR$ satisfies
      $(\ref{cond:descent})$, $(\ref{cond:deforms})$, $(\ref{cond:generizes})$,
      and $(\ref{cond:fgfieldext})$.
  \end{enumerate}
  If the closed fiber of $\varphi$ is geometrically $\bR$, then all fibers of
  $\varphi$ are geometrically $\bR$.
\end{theorem}
\par The proof in \cite{Mar84} relies on the existence of resolutions of
singularities, which we avoid by using Gabber's weak local uniformization
theorem \ref{thm:gabberunif}.
As far as we are aware, the idea to use alterations instead of resolutions of
singularities first appeared in \cite[Rem.\ 6.7]{Has01}, where Hashimoto proves
Grothendieck's localization problem \ref{prob:grothlocprob} for $F$-rationality
when the base ring $A$
is essentially of finite type over a field of positive characteristic
using de Jong's alteration theorem \cite[Thm.\ 4.1]{dJ96}.
When $\bR =$ Cohen--Macaulay or $(S_n)$, one can use Kawasaki's
Macaulayfication theorem \cite[Thm.\ 1.1]{Kaw02} to prove Theorem
\ref{thm:grothlocprobqe}; see
\citeleft\citen{BI84}\citemid Prop.\ 3.1\citepunct \citen{Ion08}\citemid Thm.\
4.1 and Rem.\ 4.2\citeright.\medskip
\par Our strategy will be to ultimately reduce to the following version of
Problem \ref{prob:grothlocprob} for regular bases.
This statement corresponds to $(\ref{strategy:regular})$ in \S\ref{sect:intro}.
\begin{citedlem}[{\cite[Lem.\ 7.5.1.1]{EGAIV2}}]\label{lem:egaiv7511}
  Fix a full subcategory $\sC$ of the category of noetherian local rings that
  is stable under quotients.
  Let $\bR$ be a property of noetherian local rings satisfying
  $(\ref{cond:deforms})$, and consider a flat local homomorphism $\varphi\colon
  (C,\fm,k) \to (D,\fn,l)$ of noetherian local rings in $\sC$, where $C$ is
  regular.
  If $D \otimes_C k$ satisfies $\bR$, then $D$ satisfies $\bR$.
\end{citedlem}
\par We now prove Theorem \ref{thm:grothlocprobqe}.
\begin{proof}[Proof of Theorem \ref{thm:grothlocprobqe}]
  We want to show that for every prime ideal $\fp \subseteq A$, the
  $\kappa(\fp)$-algebra $B \otimes_A \kappa(\fp)$ is geometrically $\bR$ over
  $\kappa(\fp)$.
  By noetherian induction, it suffices to show that if $\fp_0 \subseteq A$ is a
  prime ideal and $B \otimes_A \kappa(\fp)$ is geometrically $\bR$ over
  $\kappa(\fp)$ for every prime ideal $\fp \supsetneq \fp_0$, then
  $B \otimes_A \kappa(\fp_0)$ is geometrically $\bR$ over $\kappa(\fp_0)$.
  Replacing $A$ by $A/\fp_0$ and $B$ by $B/\fp_0 B$, we may therefore assume
  that $A$ is a domain and that $B \otimes_A \kappa(\fp)$ is geometrically
  $\bR$ over $\kappa(\fp)$ for every nonzero prime ideal $\fp \subseteq A$.
  We note that $A/\fp_0$ is quasi-excellent by \cite[Prop.\
  7.3.15$(i)$]{EGAIV2},
  and that $B/\fp_0 B$ appears in $\sC$ by $(\ref{thm:qecatcond})$.
  \par We now reduce to showing that the local rings of $B \otimes_A K$ satisfy
  $\bR$, where $K \coloneqq \Frac(A)$.
  We want to show that for every
  finite field extension $K \subseteq K'$, the local rings of $B \otimes_A K'$
  satisfy $\bR$.
  Let $A \subseteq A'$ be a module-finite extension such that $A'$ is a domain
  and such that $K' = \Frac(A')$.
  Then, $A'$ is a semi-local ring, and setting $B' \coloneqq B \otimes_A A'$, we
  have $B \otimes_A K' \simeq B' \otimes_{A'} K'$.
  Note that every local ring of $B \otimes_A K'$ is a local ring of $B'_\fq
  \otimes_{A'_\fp} K'$ for some prime ideals $\fp \subseteq A'$ and $\fq
  \subseteq B'$.
  Moreover, denoting by $k'$ the residue field of $A'_\fp$, we have
  \[
    B \otimes_A k' \simeq B \otimes_A A' \otimes_{A'} k' \simeq B'
    \otimes_{A'} k'.
  \]
  Thus, $B'_\fq \otimes_{A'_\fp} k'$ is a localization of $B \otimes_A k'$,
  and hence $B'_\fq \otimes_{A'_\fp} k'$ is geometrically $\bR$ over $k'$.
  We may therefore replace $A$ by $A'_\fp$, $B$ by $B'_\fq$, and
  $K$ by $K'$, in which case it suffices to show that the local rings of
  $B \otimes_A K$ satisfy $\bR$.
  We note that $A'_\fp$ is quasi-excellent by \cite[Prop.\
  7.3.15$(i)$ and Thm.\ 7.7.2]{EGAIV2},
  and that $B'_\fq$ appears in $\sC$ by $(\ref{thm:qecatcond})$.
  \par We now apply Gabber's weak local uniformization theorem
  \ref{thm:gabberunif}.
  Since $X\coloneqq \Spec(A)$ is quasi-excellent by $(\ref{thm:qecond})$,
  Theorem \ref{thm:gabberunif} implies there exists a finite covering
  $\{f_i\colon Y_i \to X\}_{i=1}^m$ of $X$ in the alteration topology, where
  $Y_i$ is regular and integral for every $i$.
  Note that $X$ is irreducible by our reduction in the first paragraph and by
  construction of $A'$ in the previous paragraph.
  Thus, Theorem \ref{thm:ksdefaltcover} implies there exists a proper surjective
  morphism $\pi\colon V \to X$ and a Zariski open covering $V =
  \bigcup_{i=1}^m V_i$ fitting into a commutative diagram
  \[
    \begin{tikzcd}
      V_i \rar[hook]\dar[swap]{h_i} & V\dar{\pi}\\
      Y_i \rar{f_i} & X
    \end{tikzcd}
  \]
  for every $i$.
  By base change along the morphism $\Spec(\varphi)\colon X' \to X$, where $X'
  \coloneqq \Spec(B)$, we obtain the commutative diagram
  \begin{equation}\label{eq:appgabberunif}
    \begin{tikzcd}
      V_i' \rar[hook]\dar[swap]{h_i'} & V'\dar{\pi'}\\
      Y'_i \rar{f'_i}\dar[swap]{g_i} & X'\dar{\mathrlap{\Spec(\varphi)}}\\
      Y_i \rar{f_i} & X
    \end{tikzcd}
  \end{equation}
  with cartesian squares for every $i$.
  We now consider a point $\eta' \in X'$ lying over the generic point of $X$.
  We want to show that $\cO_{X',\eta'}$ satisfies $\bR$.
  Since $\pi'$ is surjective, there exists a point $\xi' \in V'$ such that
  $\pi'(\xi') = \eta'$.
  Since $\pi'$ is closed, the specialization $\eta' \rightsquigarrow \fn$ in $X'
  = \Spec(B)$ then lifts to a specialization $\xi' \rightsquigarrow v'$ in $V'$
  \cite[\href{https://stacks.math.columbia.edu/tag/0066}{Tag
  0066}]{stacks-project}.
  Since $V' = \bigcup_{i=1}^m V'_i$ is a Zariski open covering, there exists an
  index $i_0 \in \{1,2,\ldots,m\}$ such that $v' \in V_{i_0}'$, and since open
  sets are stable under generization, we have $\xi' \in V_{i_0}'$ as well.
  We claim that it suffices to show that $\cO_{Y_{i_0}',h_{i_0}'(\xi')}$
  satisfies $\bR$.
  Since the morphism $f_{i_0}$ is maximally dominating, base changing the bottom
  square in \eqref{eq:appgabberunif} for $i = i_0$ along the morphism $\Spec(K)
  \to X = \Spec(A)$, localizing at the generic point of $Y_{i_0}$,
  and taking global sections yields the cocartesian square
  \[
    \begin{tikzcd}
      B \otimes_A L_{i_0} & \lar B \otimes_A K\\
      L_{i_0} \uar & \lar[hook'] K \uar
    \end{tikzcd}
  \]
  of rings, where $L_{i_0}$ is the function field of $Y_{i_0}$.
  The bottom horizontal arrow is faithfully flat; thus, the top horizontal
  arrow is also faithfully flat by base change.
  After localizing, we therefore obtain a faithfully flat homomorphism
  \[
    \cO_{X',\eta'} \simeq (B \otimes_A K)_{\eta'} \longrightarrow (B \otimes_A
    L_{i_0})_{h_{i_0}'(\xi')} \simeq \cO_{Y_{i_0}',h_{i_0}'(\xi')},
  \]
  and $(\ref{cond:descent})$ implies that if $\cO_{Y_{i_0}',h_{i_0}'(\xi')}$
  satisfies $\bR$, then $\cO_{X',\eta'}$ satisfies $\bR$.
  Here we use the fact that $(B \otimes_A L_{i_0})_{h_{i_0}'(\xi')}$ is in $\sC$
  by $(\ref{thm:qecatcond})$.
  \par It remains to show that $\cO_{Y_{i_0}',h_{i_0}'(\xi')}$ satisfies $\bR$.
  Setting $y' \coloneqq h_{i_0}'(v')$ and $y \coloneqq g_{i_0}(y')$, the residue
  field
  extension $k \subseteq \kappa(y)$ is finitely generated since $f_{i_0}$ is of
  finite type.
  Thus, the closed fiber of the flat homomorphism $\psi$ in the commutative
  diagram
  \[
    \begin{tikzcd}
      \cO_{Y_{i_0}',y'} & \lar B\\
      \cO_{Y_{i_0},y} \uar{\psi} & \lar A \uar[swap]{\varphi}
    \end{tikzcd}
  \]
  is geometrically $\bR$ by $(\ref{cond:fgfieldext})$, since $B \otimes_A k$ is
  in $\sC$ by $(\ref{thm:qecatcond})$.
  Since $\cO_{Y_{i_0},y}$ is regular by construction and $\cO_{Y'_{i_0},y'}$
  appears in
  $\sC$ by $(\ref{thm:qecatcond})$, we can apply Lemma \ref{lem:egaiv7511}
  (which uses $(\ref{cond:deforms})$) to deduce that $\cO_{Y'_{i_0},y'}$
  satisfies
  $\bR$.
  Finally, $(\ref{cond:generizes})$ implies that $\cO_{Y_{i_0}',h_{i_0}'(\xi')}$
  satisfies $\bR$, since the specialization $\xi' \rightsquigarrow v'$ in $V'$
  maps to the specialization $h_{i_0}'(\xi') \rightsquigarrow h_{i_0}'(v') = y'$
  by continuity.
\end{proof}
\subsection{Problem \ref{prob:grothlocprob} in general and the proof of Theorem
\ref{thm:grothlocprob}}
\label{sect:grothlocprobpring}
We now prove Theorem \ref{thm:grothlocprob} by reducing to the complete (hence
quasi-excellent) case proved in Theorem \ref{thm:grothlocprobqe}.
This step corresponds to $(\ref{strategy:complete})$ in \S\ref{sect:intro}.
\par We first show the following stronger statement that is more specific
about what conditions from Conditions \ref{cond:listof} are needed.
\begin{corollary}[cf.\ {\citeleft\citen{Mar84}\citemid Thm.\ 2.2\citepunct
  \citen{BI84}\citemid Prop.\ 1.2\citeright}]\label{thm:grothlocprobpring}
  Fix a full subcategory $\sC$ of the category of noetherian rings that is
  stable under homomorphisms essentially of finite type.
  Let $\bR$ be a property of noetherian local rings, and consider a flat local
  homomorphism $\varphi\colon (A,\fm,k) \to (B,\fn,l)$ of
  noetherian local rings.
  Assume the following:
  \begin{enumerate}[label=$(\roman*)$,ref=\roman*]
    \item\label{thm:pringpcond} The ring $A$ has geometrically $\bR$ formal
      fibers;
    \item\label{thm:pringcatcond} The rings $\widehat{A}$ and $B^*$ appear in
      $\sC$, where $\widehat{A}$ and $B^*$ denote the $\fm$-adic completions of
      $A$ and $B$, respectively; and
    \item\label{thm:pringrconds} The property $\bR$ satisfies
      $(\ref{cond:descent})$, $(\ref{cond:deforms})$, $(\ref{cond:generizes})$,
      $(\ref{cond:transitivity})$,
      and $(\ref{cond:fgfieldext})$.
  \end{enumerate}
  If the closed fiber of $\varphi$ is geometrically $\bR$, then all fibers of
  $\varphi$ are geometrically $\bR$.
\end{corollary}
\begin{proof}
  We have the commutative square
  \[
    \begin{tikzcd}
      A \rar{\varphi}\dar[swap]{\sigma} & B\dar{\tau}\\
      \widehat{A} \rar{\varphi^*} & B^*
    \end{tikzcd}
  \]
  where $\sigma$ and $\tau$ are the canonical $\fm$-adic completion
  homomorphisms.
  By $(\ref{thm:pringpcond})$, the homomorphism $\sigma$ is geometrically $\bR$.
  We claim that $\varphi^*$ is geometrically $\bR$.
  By \cite[Ch.\ 0, Lem.\ 6.8.3.1]{EGAInew}, the ring $B^*$ is a noetherian local
  ring, and we have
  \[
    B^* \otimes_{\widehat{A}} k \simeq B \otimes_A k.
  \]
  Theorem \ref{thm:grothlocprobqe} therefore implies that
  $\varphi^*$ is geometrically $\bR$, where we use the fact that the
  complete local ring $\widehat{A}$ is excellent by \cite[Sch.\
  7.8.3$(iii)$]{EGAIV2}.
  The composition $\tau \circ \varphi =
  \varphi^* \circ \sigma$ is geometrically $\bR$ by
  $(\ref{cond:transitivity})$, and therefore $\varphi$ is also by
  $(\ref{cond:descentphom})$ (which holds by $(\ref{cond:descent})$ and Lemma
  \ref{lem:egariimpliespj}$(\ref{lem:egar2impliesp2})$) since $\tau$ is
  faithfully flat \cite[Thm.\ 8.14]{Mat89}.
\end{proof}
We now deduce Theorem \ref{thm:grothlocprob} as a consequence.
\begin{customthm}{B}\label{thm:grothlocprob}
  Grothendieck's localization problem \ref{prob:grothlocprob} holds for
  properties $\bR$ of noetherian local rings satisfying
  $(\hyperref[cond:ascentphom]{\textup{R}'_{\textup{I}}})$,
  $(\hyperref[cond:descent]{\textup{R}_{\textup{II}}})$,
  $(\hyperref[cond:deforms]{\textup{R}_{\textup{IV}}})$, and
  $(\hyperref[cond:generizes]{\textup{R}_{\textup{V}}})$, such that regular
  local rings satisfy $\bR$.
\end{customthm}
\begin{proof}
  It suffices to show that the hypotheses in Theorem \ref{thm:grothlocprob}
  imply those in Corollary \ref{thm:grothlocprobpring} when $\sC$ is the
  entire category of noetherian rings.
  Note that $(\ref{thm:pringpcond})$ is already a hypothesis in Theorem
  \ref{thm:grothlocprob} and that $(\ref{thm:pringcatcond})$ is vacuously true.
  It therefore suffices to note that
  $(\hyperref[cond:ascentphom]{\textup{R}'_{\textup{I}}})$ implies
  $(\hyperref[cond:transitivity]{\textup{P}'_{\textup{I}}})$
  by Lemma \ref{lem:egariimpliespj}$(\ref{lem:egar1impliesp1})$,
  and that
  $(\hyperref[cond:ascentreghom]{\textup{R}_{\textup{I}}})$ and
  $(\hyperref[cond:descent]{\textup{R}_{\textup{II}}})$ imply
  $(\hyperref[cond:fgfieldext]{\textup{P}_{\textup{IV}}})$
  by Lemma \ref{lem:egariimpliespj}$(\ref{lem:egar1impliesp4})$.
  Here,
  $(\hyperref[cond:ascentreghom]{\textup{R}_{\textup{I}}})$
  holds by
  $(\hyperref[cond:ascentphom]{\textup{R}'_{\textup{I}}})$
  and the assumption that regular local rings satisfy $\bR$.
\end{proof}
We can also prove a version of Corollary \ref{thm:grothlocprobpring} for a
specific choice of the category $\sC$, as long as we put an extra condition
on $B \otimes_A k$.
\begin{corollary}[cf.\ {\citeleft\citen{EGAIV2}\citemid Cor.\ 7.5.2\citepunct
  \citen{BI84}\citemid Prop.\ 1.2\citeright}]
  \label{thm:grothlocprobbothpring}
  Denote by $\sC$ the smallest full subcategory of the category of noetherian
  rings containing noetherian complete local rings that is stable under
  homomorphisms essentially of finite type.
  Let $\bR$ be a property of noetherian local rings, and consider a flat local
  homomorphism $\varphi\colon (A,\fm,k) \to (B,\fn,l)$ of
  noetherian local rings.
  Assume the following:
  \begin{enumerate}[label=$(\roman*)$,ref=\roman*]
    \item\label{thm:bothpringpcond}
      The rings $A$ and $B \otimes_A k$ have geometrically $\bR$ formal fibers;
      and
    \item\label{thm:bothpringrconds} The property $\bR$ satisfies
      $(\hyperref[cond:ascentphom]{\textup{R}'_{\textup{I}}})$,
      $(\ref{cond:descent})$, $(\ref{cond:deforms})$, $(\ref{cond:generizes})$,
      and $(\ref{cond:fgfieldext})$.
  \end{enumerate}
  If the closed fiber of $\varphi$ is geometrically $\bR$, then all fibers of
  $\varphi$ are geometrically $\bR$.
\end{corollary}
\begin{proof}
  We have the commutative square
  \[
    \begin{tikzcd}
      A \rar{\varphi}\dar[swap]{\sigma} & B\dar{\tau}\\
      \widehat{A} \rar{\widehat{\varphi}} & \widehat{B}
    \end{tikzcd}
  \]
  where $\sigma$ and $\tau$ are the canonical $\fm$-adic and $\fn$-adic
  completion homomorphisms, respectively.
  By $(\ref{thm:bothpringpcond})$, the homomorphism $\sigma$ is geometrically
  $\bR$.
  We claim that $\widehat{\varphi}$ is geometrically $\bR$.
  Note that $\widehat{\varphi}$ is flat by \cite[Thm.\ 22.4$(i)$]{Mat89}, and
  that $B \otimes_A k$ has geometrically $\bR$ formal fibers by
  $(\ref{thm:bothpringpcond})$.
  Thus, the composition
  \[
    k \longrightarrow B \otimes_A k \longrightarrow
    \widehat{B} \otimes_{\widehat{A}} k
  \]
  is geometrically $\bR$ by applying
  $(\hyperref[cond:transitivity]{\textup{P}'_{\textup{I}}})$ (which holds
  by $(\hyperref[cond:ascentphom]{\textup{R}'_{\textup{I}}})$
  and Lemma
  \ref{lem:egariimpliespj}$(\ref{lem:egar1impliesp1})$).
  Since this composition is equal to $\widehat{\varphi} \otimes_{\widehat{A}}
  \id_k$, Theorem \ref{thm:grothlocprobqe} implies that $\widehat{\varphi}$ is
  geometrically $\bR$, where we use
  the fact that the complete local ring $\widehat{A}$ is excellent by
  \cite[Sch.\ 7.8.3$(iii)$]{EGAIV2}.
  \par The composition $\tau \circ \varphi =
  \varphi^* \circ \sigma$ is geometrically $\bR$ by
  $(\hyperref[cond:transitivity]{\textup{P}'_{\textup{I}}})$, and therefore
  $\varphi$ is also by
  $(\ref{cond:descentphom})$ (which holds by $(\ref{cond:descent})$ and Lemma
  \ref{lem:egariimpliespj}$(\ref{lem:egar2impliesp2})$) since $\tau$ is
  faithfully flat \cite[Thm.\ 8.14]{Mat89}.
\end{proof}
\subsection{Global applications and the proof of Theorem
\ref{thm:grothlocprobglobalintro}}\label{sect:globalapps}
We now prove Theorem \ref{thm:grothlocprobglobalintro} by reducing to the local
statements proved above.
This step corresponds to $(\ref{strategy:localq})$ in \S\ref{sect:intro}.
\par We first give global versions of Theorem \ref{thm:grothlocprobqe} and
Corollaries \ref{thm:grothlocprobpring} and \ref{thm:grothlocprobbothpring}.
Theorem
\ref{thm:grothlocprobglobalintro}$(\ref{thm:grothlocprobglobalintroclosedfib})$
will be deduced from $(\ref{thm:grothlocprobpringglobal})$ below.
These results are related to a theorem of Shimomoto
\cite[Main Thm.\ 1]{Shi17}, which applies to morphisms of finite type between
excellent noetherian schemes.
\begin{proposition}[cf.\ {\cite[Main Thm.\ 1]{Shi17}}]\label{thm:grothlocprobglobal}
  Fix a full subcategory
  $\sC$ of the category of noetherian rings that is stable under homomorphisms
  essentially of finite type.
  Let $\bR$ be a property of noetherian local rings, and
  consider a flat morphism $f\colon Y \to X$ of locally noetherian
  schemes mapping closed points to closed points.
  Assume one of the following:
  \begin{enumerate}[label=$(\roman*)$,ref=\roman*]
    \item $\bR$ satisfies the hypotheses of Theorem \ref{thm:grothlocprobqe} for
      the category $\sC$, the local rings of $X$ at closed points are
      quasi-excellent, and the local rings of $Y$
      at closed points appear in $\sC$;\label{thm:grothlocprobqeglobal}
    \item $\bR$ satisfies the hypotheses of Corollary \ref{thm:grothlocprobpring}
      for the category $\sC$, the local rings of $X$ at closed points have
      geometrically $\bR$ formal fibers, and the rings $\widehat{\cO}_{X,f(y)}$
      and $\cO_{Y,y}^*$ appear in $\sC$ for every closed point $y \in Y$, where
      $\cO_{Y,y}^*$ denotes the
      $\fm_{f(y)}$-adic completion of $\cO_{Y,y}$;
      or\label{thm:grothlocprobpringglobal}
    \item $\bR$ satisfies the hypotheses of Corollary
      \ref{thm:grothlocprobbothpring} for the specific choice of $\sC$ therein,
      the local rings of $X$ at closed points have geometrically $\bR$ formal
      fibers,
      and the local rings $\cO_{f^{-1}(x),y}$
      have geometrically $\bR$ formal
      fibers for every closed point $x \in X$ and every closed
      point $y \in f^{-1}(x)$.\label{thm:grothlocprobbothpringglobal}
  \end{enumerate}
  If every closed fiber of $f$ is geometrically $\bR$, then all fibers of $f$
  are geometrically $\bR$.
\end{proposition}
\begin{proof}
  Let $\xi \in Y$ be an arbitrary point, and let $\eta = f(\xi)$.
  We want to show that $\cO_{Y,\xi}$ is geometrically $\bR$ over $\kappa(\eta)$.
  By \cite[\href{https://stacks.math.columbia.edu/tag/02IL}{Tag
  02IL}]{stacks-project}, the point $\xi$ specializes to a closed point $y \in
  Y$.
  Since $f$ maps closed points to closed points, the point $x = f(y)$ is closed
  in $X$.
  After localization, we then obtain a flat local homomorphism
  \[
    \varphi\colon \cO_{X,x} \longrightarrow \cO_{Y,y}
  \]
  whose closed fiber is geometrically $\bR$ by assumption.
  Finally, we can apply Theorem \ref{thm:grothlocprobqe}, Corollary
  \ref{thm:grothlocprobpring}, and Corollary \ref{thm:grothlocprobbothpring} to
  the homomorphism $\varphi$ under the assumptions in
  $(\ref{thm:grothlocprobqeglobal})$, $(\ref{thm:grothlocprobpringglobal})$, and
  $(\ref{thm:grothlocprobbothpringglobal})$, respectively, to conclude that
  $\cO_{Y,\xi}$ is geometrically $\bR$ over $\kappa(\eta)$.
\end{proof}
\par We also prove that the locus over which $f$ has geometrically $\bR$
fibers is stable under generization for closed flat morphisms.
Theorem
\ref{thm:grothlocprobglobalintro}$(\ref{thm:grothlocprobglobalintrogenerizes})$
will be deduced from $(\ref{cor:grothlocprobpringglobal})$ below.
\begin{proposition}\label{cor:pmorphismgenerizes}
  Fix a full subcategory
  $\sC$ of the category of noetherian rings that is stable under homomorphisms
  essentially of finite type.
  Let $\bR$ be a property of noetherian local rings, and consider a closed flat
  morphism $f\colon Y \to X$ of locally noetherian schemes.
  Assume one of the following:
  \begin{enumerate}[label=$(\roman*)$,ref=\roman*]
    \item $\bR$ satisfies the hypotheses of Theorem \ref{thm:grothlocprobqe} for
      the category $\sC$, the local rings of $X$ are quasi-excellent, and the
      local rings of $Y$ appear in $\sC$;\label{cor:grothlocprobqeglobal}
    \item $\bR$ satisfies the hypotheses of Corollary \ref{thm:grothlocprobpring}
      for the category $\sC$, the local rings of $X$ have geometrically $\bR$
      formal fibers, and the rings $\widehat{\cO}_{X,f(y)}$ and $\cO_{Y,y}^*$
      appear in $\sC$ for every $y \in Y$, where $\cO_{Y,y}^*$ denotes the
      $\fm_{f(y)}$-adic completion of $\cO_{Y,y}$;
      or\label{cor:grothlocprobpringglobal}
    \item $\bR$ satisfies the hypotheses of Corollary
      \ref{thm:grothlocprobbothpring} for the specific choice of $\sC$ therein,
      the local rings of $X$ have geometrically $\bR$ formal fibers, and the
      local rings $\cO_{f^{-1}(x),y}$
      have geometrically $\bR$ formal fibers for every $x \in X$ and every
      $y \in f^{-1}(x)$.\label{cor:grothlocprobbothpringglobal}
  \end{enumerate}
  Then, the locus
  \[
    U_{\bR}(f) \coloneqq \Set[\big]{x \in X \given f^{-1}(x)\ \text{is
    geometrically}\ \bR\ \text{over}\ \kappa(x)} \subseteq X
  \]
  is stable under generization.
\end{proposition}
\begin{proof}
  Consider a specialization $\eta \rightsquigarrow x$ in $X$, and suppose that
  $f^{-1}(x)$ is geometrically $\bR$ over $\kappa(x)$.
  We want to show that $f^{-1}(\eta)$ is geometrically $\bR$ over
  $\kappa(\eta)$.
  It suffices to show that for every $\xi \in f^{-1}(\eta)$, the local ring
  $\cO_{Y,\xi}$ is geometrically $\bR$ over $\kappa(\eta)$.
  Since $f$ is closed, the specialization $\eta \rightsquigarrow x$ lifts to a
  specialization $\xi \rightsquigarrow y$
  \cite[\href{https://stacks.math.columbia.edu/tag/0066}{Tag
  0066}]{stacks-project}.
  After localization, we then obtain a flat local homomorphism
  \[
    \varphi\colon \cO_{X,x} \longrightarrow \cO_{Y,y}
  \]
  whose closed fiber is geometrically $\bR$ by assumption.
  Finally, we can apply Theorem \ref{thm:grothlocprobqe}, Corollary
  \ref{thm:grothlocprobpring}, and Corollary \ref{thm:grothlocprobbothpring} to
  the homomorphism $\varphi$ under the assumptions in
  $(\ref{cor:grothlocprobqeglobal})$, $(\ref{cor:grothlocprobpringglobal})$, and
  $(\ref{cor:grothlocprobbothpringglobal})$, respectively, to conclude that
  $\cO_{Y,\xi}$ is geometrically $\bR$ over $\eta$.
\end{proof}
\par We now deduce Theorem \ref{thm:grothlocprobglobalintro} as a consequence of
Propositions \ref{thm:grothlocprobglobal} and \ref{cor:pmorphismgenerizes}.
\begin{customthm}{A}\label{thm:grothlocprobglobalintro}
  Let $\bR$ be a property of noetherian local rings satisfying
  $(\hyperref[cond:ascentphom]{\textup{R}'_{\textup{I}}})$,
  $(\hyperref[cond:descent]{\textup{R}_{\textup{II}}})$,
  $(\hyperref[cond:deforms]{\textup{R}_{\textup{IV}}})$, and
  $(\hyperref[cond:generizes]{\textup{R}_{\textup{V}}})$, such that regular
  local rings satisfy $\bR$.
  Consider a flat morphism $f\colon Y \to X$ of locally noetherian
  schemes.
  \begin{enumerate}[label=$(\roman*)$,ref=\roman*]
    \item Suppose that $f$ maps closed to closed points, and that the local
      rings of $X$ at closed points have geometrically $\bR$ formal fibers.
      If every closed fiber of $f$ is geometrically $\bR$, then all fibers of
      $f$ are geometrically $\bR$.\label{thm:grothlocprobglobalintroclosedfib}
    \item Suppose that $f$ is closed, and that the local rings of $X$ have
      geometrically $\bR$ formal fibers.
      Then, the locus
      \[
        U_\bR(f) \coloneqq \Set[\big]{x \in X \given f^{-1}(x)\ \text{is
        geometrically}\ \bR\ \text{over}\ \kappa(x)} \subseteq X
      \]
      is stable under generization.\label{thm:grothlocprobglobalintrogenerizes}
  \end{enumerate}
\end{customthm}
\begin{proof}
  As in the proof of Theorem \ref{thm:grothlocprob}, it suffices to note that
  the hypotheses in $(\ref{thm:grothlocprobglobalintroclosedfib})$ and
  $(\ref{thm:grothlocprobglobalintrogenerizes})$
  imply those in Propositions
  \ref{thm:grothlocprobglobal}$(\ref{thm:grothlocprobpringglobal})$ and
  \ref{cor:pmorphismgenerizes}$(\ref{cor:grothlocprobpringglobal})$,
  respectively, when $\sC$ is the entire category of noetherian rings.
\end{proof}

\subsection{The local lifting problem \ref{prob:grothliftprob}}
\label{sect:liftingproblem}
To solve the local lifting problem \ref{prob:grothliftprob}, we use the
following theorem of Brezuleanu and Ionescu.
We state their theorem using the notation in Conditions \ref{cond:listof}
instead of that in \cite[(2.1)]{BI84}.
\begin{citedthm}[{\cite[Thm.\ 2.3]{BI84}}]\label{thm:bi23}
  Let $\bR'$ be a property of noetherian local rings.
  Assume the following:
  \begin{enumerate}[label=$(\roman*)$,ref=\roman*]
    \item\label{thm:bi23reg} We have the following sequence of implications:
      regular $\Rightarrow$ $\bR'$ $\Rightarrow$
      reduced.
    \item\label{thm:bi23conds} The property $\bR'$ satisfies 
      $(\hyperref[cond:ascentphom]{\textup{R}'_{\textup{I}}})$,
      $(\hyperref[cond:descent]{\textup{R}_{\textup{II}}})$,
      $(\ref{cond:openness})$,
      and
      $(\hyperref[cond:generizes]{\textup{R}_{\textup{V}}})$.
    \item\label{thm:bi23locprob}
      For every flat local homomorphism $\varphi\colon A \to B$ of noetherian
      complete local rings, if the closed fiber of $\varphi$ is geometrically
      $\bR'$, then all fibers of $\varphi$ are geometrically $\bR'$.
  \end{enumerate}
  Let $A$ be a noetherian semi-local ring that is $I$-adically complete with
  respect to an ideal $I \subseteq A$.
  If $A/I$ is has geometrically $\bR'$ formal fibers, then $A$ has geometrically
  $\bR'$ formal fibers.
\end{citedthm}
We obtain Theorem \ref{thm:grothlocliftprob} and an important corollary as
immediate consequences.
\begin{customthm}{C}\label{thm:grothlocliftprob}
  Let $\bR$ be a property of noetherian local rings that satisfies the
  hypotheses in Theorem \ref{thm:grothlocprob}.
  Suppose, moreover, that the locus
  \[
    U_\bR\bigl(\Spec(C)\bigr) \coloneqq
    \Set[\big]{\fp \in \Spec(C) \given C_\fp\ \text{satisfies $\bR$}}
    \subseteq \Spec(C)
  \]
  is open for every noetherian complete local ring $C$.
  If $A$ is a noetherian semi-local ring that is $I$-adically complete with
  respect to an ideal $I \subseteq A$, and if $A/I$ is Nagata and 
  has geometrically $\bR$ formal fibers, then $A$ is Nagata and has
  geometrically $\bR$ formal fibers.
\end{customthm}
\begin{proof}
  Since a noetherian semi-local ring is Nagata if and only if it has
  geometrically reduced formal fibers (Remark \ref{rem:zariskinagata}),
  it suffices to verify the hypotheses in Theorem \ref{thm:bi23} for
  $\bR' = \bR + \text{``reduced.''}$
  Both $(\ref{thm:bi23reg})$ and $(\ref{thm:bi23conds})$ hold by assumption.
  Theorem \ref{thm:grothlocprob} implies $(\ref{thm:bi23locprob})$
  holds, since complete local rings have geometrically $\bR$ formal fibers by
  the assumption that regular local rings satisfy $\bR$.
\end{proof}
\begin{corollary}\label{cor:grothlocliftprob}
  With assumptions as in Theorem \ref{thm:grothlocliftprob}, if $B$ is a
  noetherian semi-local Nagata ring that has geometrically $\bR$ formal fibers,
  then for every ideal $I \subseteq B$,
  the $I$-adic completion of $B$ is Nagata and has geometrically $\bR$ formal
  fibers.
\end{corollary}
\begin{proof}
  The $I$-adic completion $\widehat{B}$ of $B$ is a noetherian semi-local ring
  that is $I\widehat{B}$-adically complete by \cite[Ch.\ III,
  \S3, n\textsuperscript{o} 4, Prop.\ 8]{BouCA}.
  Thus, by Theorem \ref{thm:grothlocliftprob}, it suffices to show that
  $\widehat{B}/I\widehat{B} \simeq B/I$ is a Nagata ring with geometrically
  $\bR$ formal fibers.
  But $B/I$ has geometrically $\bR$ formal fibers by \cite[Prop.\
  7.3.15$(i)$]{EGAIV2}, and is also Nagata by Definition \ref{def:nagata}.
\end{proof}
\section{Specific properties \texorpdfstring{$\bR$}{R}}\label{sect:specific}
We now explicitly consider our new cases of Grothendieck's localization problem
\ref{prob:grothlocprob} and the local lifting problem \ref{prob:grothliftprob}.
We have listed known cases of conditions 
$(\hyperref[cond:ascentphom]{\textup{R}'_{\textup{I}}})$,
$(\hyperref[cond:descent]{\textup{R}_{\textup{II}}})$,
$(\ref{cond:openness})$,
$(\hyperref[cond:deforms]{\textup{R}_{\textup{IV}}})$, and
$(\hyperref[cond:generizes]{\textup{R}_{\textup{V}}})$
in Table \ref{table:condslist}.
While Problem \ref{prob:grothlocprob} follows readily from these results
when $\bR =$ ``domain,'' ``Cohen--Macaulay and $F$-injective,'' and
``$F$-rational,'' we will have to verify
$(\hyperref[cond:deforms]{\textup{R}_{\textup{IV}}})$ for weak normality
(Proposition \ref{prop:weaklynormaldeforms}).
We will deduce our results for terminal, canonical, and rational singularities
from \cite[Prop.\ 7.9.8]{EGAIV2} instead of Theorems
\ref{thm:grothlocprobglobalintro} and \ref{thm:grothlocprob}.
\par We will not explicitly formulate versions of Theorem
\ref{thm:grothlocprobglobalintro} below, except for terminal, canonical, and
rational singularities
(Corollary \ref{cor:rationalsingglobal}).
\subsection{Domain}
We first note that a version of Problem \ref{prob:grothlocprob} holds for
$\bR =$ ``domain.''
This property satisfies the hypotheses in Theorem
\ref{thm:grothlocprobqe} (see Table \ref{table:condslist}).
This extends a result of Marot \cite{Mar84}, which holds in residue
characteristic zero.
We use the terminology ``geometrically punctually integral'' following
\cite[D\'ef.\ 4.6.9]{EGAIV2} for the property obtained by applying Definition
\ref{def:pring} to $\bR =$ ``domain.''
\begin{corollary}[{cf.\ \cite[Thm.\ 2.1]{Mar84}}]\label{cor:domain}
  Let $\varphi\colon A \to B$ be a flat local homomorphism of
  noetherian local rings, and assume that $A$ is quasi-excellent.
  If the closed fiber of $\varphi$ is geometrically punctually integral, then
  all fibers of $\varphi$ are geometrically punctually integral.
\end{corollary}
\begin{remark}
  Since $(\ref{cond:ascentphom})$ is false for the property $\bR =$ ``domain''
  \cite[Rems.\ 6.5.5$(ii)$ and 6.15.11$(ii)$]{EGAIV2},
  we only know that global versions of Corollary
  \ref{cor:domain} hold under quasi-excellence assumptions by applying
  Propositions \ref{thm:grothlocprobglobal}$(\ref{thm:grothlocprobqeglobal})$
  and \ref{cor:pmorphismgenerizes}$(\ref{cor:grothlocprobqeglobal})$.
  See also \cite[Cor.\ 7.9.9 and Rem.\ 7.9.10$(i)$]{EGAIV2}.
\end{remark}
\subsection{\emph{F}-singularities}
We now solve Problems \ref{prob:grothlocprob} and \ref{prob:grothliftprob} for
``Cohen--Macaulay and $F$-injective'' and for $F$-rationality.
See \cite[Def.\ on p.\ 473]{Fed83} and \cite[Def.\ 4.1]{HH94} for the
definitions of $F$-injectivity and $F$-rationality, respectively.
We recall (see \cite[Rem.\ A.4]{DM}) that we have the following sequence of
implications:
\begin{equation}\label{eq:fsingsimplications}
  \text{regular} \Longrightarrow \text{$F$-rational} \Longrightarrow
  \text{$F$-injective} \Longrightarrow \text{reduced}
\end{equation}
\par Since the property ``Cohen--Macaulay and $F$-injective'' satisfies the
hypotheses in Theorems \ref{thm:grothlocprob} and \ref{thm:grothlocliftprob}
(see Table \ref{table:condslist} and \eqref{eq:fsingsimplications}),
we can solve Problems \ref{prob:grothlocprob} and \ref{prob:grothliftprob} for
this property.
The result for Problem \ref{prob:grothlocprob} extends a result of Hashimoto
\cite{Has01} to the non-$F$-finite case.
\begin{corollary}[cf.\ {\cite[Thm.\ 5.8]{Has01}}]\label{cor:applythmacmfi}
  Grothendieck's localization problem \ref{prob:grothlocprob} and the local
  lifting problem \ref{prob:grothliftprob} hold for
  ``Cohen--Macaulay and $F$-injective,'' where $A$ is assumed to be of prime
  characteristic $p > 0$.
\end{corollary}
\begin{remark}
  Shimomoto and Zhang proved Grothendieck's localization problem
  \ref{prob:grothlocprob} for the property ``Gorenstein and $F$-pure'' when $A$
  and $B$ are $F$-finite \cite[Thm.\ 3.10]{SZ09}.
  Their result is a special case of Hashimoto's result, since $F$-purity and
  $F$-injectivity coincide for Gorenstein rings \cite[Lem.\ 3.3]{Fed83}, and
  since Problem \ref{prob:grothlocprob} holds for Gorensteinness \cite[Thm.\
  3.2]{Mar84}.
  Corollary \ref{cor:applythmacmfi} also extends Shimomoto and Zhang's result to
  the non-$F$-finite case.
  See also Remark \ref{rem:psz}.
\end{remark}
\par Next, we consider Grothendieck's localization problem
\ref{prob:grothlocprob} for $F$-rationality.
The following extends a result of Hashimoto \cite{Has01} to rings $A$ not
necessarily essentially of finite type over a field, and a result of Shimomoto
\cite{Shi17} to homomorphisms not necessarily of finite type,
giving a complete answer to a question of Hashimoto \cite[Rem.\
6.7]{Has01}.
\begin{corollary}[cf.\ {\citeleft\citen{Has01}\citemid Rem.\
  6.7\citepunct\citen{Shi17}\citemid Cor.\ 3.10\citeright}]
  \label{cor:frational}
  Let $\varphi\colon A \to B$ be a flat local homomorphism of
  noetherian local rings of prime characteristic $p > 0$.
  Assume that $A$ is quasi-excellent and that $B$ is excellent.
  If the closed fiber of $\varphi$ is geometrically $F$-rational, then all
  fibers of $\varphi$ are geometrically $F$-rational.
\end{corollary}
\begin{proof}
  We apply Theorem \ref{thm:grothlocprobqe} when $\sC$ is the category
  of excellent rings, which is stable under homomorphisms essentially of finite
  type by \cite[Sch.\ 7.8.3$(ii)$]{EGAIV2}.
  To verify the hypotheses of Theorem \ref{thm:grothlocprobqe}, it suffices to
  note that excellent local rings are homomorphic images of Cohen--Macaulay
  rings \cite[Cor.\ 1.2]{Kaw02}, and hence $(\ref{cond:deforms})$,
  $(\ref{cond:generizes})$, and $(\ref{cond:fgfieldext})$ hold by
  Table \ref{table:condslist}.
\end{proof}
\begin{remark}\label{rem:psz}
  Patakfalvi, Schwede, and Zhang also obtained a version of Problem
  \ref{prob:grothlocprob} for proper flat morphisms $f \colon Y \to X$, showing
  that the locus $U_\bR(f)$ defined in Definition \ref{def:pring}
  is open when $\bR =$ ``Cohen--Macaulay and
  $F$-injective'' (resp.\ ``$F$-rational'') under
  the assumption that $X$ is an excellent integral scheme with a dualizing
  complex \cite[Thm.\ 5.13]{PSZ18}.
  Theorem \ref{thm:grothlocprobglobalintro}$(\ref{thm:grothlocprobglobalintrogenerizes})$
  (resp.\ Proposition
  \ref{cor:pmorphismgenerizes}$(\ref{cor:grothlocprobqeglobal})$)
  implies that the locus $U_\bR(f)$ is stable under generization, even if
  $f$ is only closed and flat (resp.\ $f$ is closed and flat, $X$
  is quasi-excellent, and $Y$ is excellent).
  \par For $F$-rationality, we note that
  Theorem
  \ref{thm:grothlocprobglobalintro}$(\ref{thm:grothlocprobglobalintroclosedfib})$
  does not apply, but one can apply Proposition
  \ref{thm:grothlocprobglobal}$(\ref{thm:grothlocprobqeglobal})$ instead.
\end{remark}

\subsection{Weak normality and seminormality}
In this subsection, we prove that weak normality lifts from Cartier divisors.
We then solve Problems \ref{prob:grothlocprob} and \ref{prob:grothliftprob} for
both weak normality and seminormality.
\par To fix notation, we first define weak normality and seminormality.
\begin{definition}[see {\cite[Def.\ 50]{Kol16}}]
  Let $X$ be a noetherian scheme.
  A morphism $g\colon X' \to X$ is a \textsl{partial normalization} if $X'$ is
  reduced, $g$ is integral, and $X' \to X_\red$ is birational.
  A partial normalization is a \textsl{partial weak normalization} if $g$ is a
  universal homeomorphism.
  A partial weak normalization is a \textsl{partial seminormalization} if
  the induced extensions of residue fields
  \begin{equation}\label{eq:extresfields}
    \kappa(x) \subseteq \kappa\bigl(g^{-1}(x)_\red\bigr)
  \end{equation}
  are bijective for all $x \in X$.
  \par We say that $X$ is \textsl{weakly normal} (resp.\ \textsl{seminormal})
  if every finite partial weak normalization (resp.\ finite partial
  seminormalization) $g\colon X' \to X$ is an isomorphism.
  A noetherian ring $A$ is \textsl{weakly normal} (resp.\ seminormal) if
  $\Spec(A)$ is weakly normal (resp.\ seminormal).
\end{definition}
\par We have the sequence of implications
\begin{equation}\label{eq:varnorimplications}
  \text{normal} \Longrightarrow \text{weakly normal} \Longrightarrow
  \text{seminormal} \Longrightarrow \text{reduced}
\end{equation}
where for the last implication, we use that $X_\red \to X$ is a finite partial
seminormalization; see \cite[(4.4)]{Kol16}.
We also note that a partial normalization is a partial weak normalization if and
only if the extensions \eqref{eq:extresfields} are
purely inseparable for all $x \in X$ by \cite[Cor.\ 18.12.11]{EGAIV4}.\medskip
\par To show that weak normality lift from Cartier divisors,
we follow the proof in \cite{BF93}, with suitable modifications to avoid
excellence hypotheses.
\begin{lemma}[cf.\ {\cite[Prop.\ 4.7]{BF93}}]\label{lem:bfprop47}
  Let $(A,\fm_A)$ be a noetherian local ring with $\depth(A_\red) \ge 2$, and
  let $B$ be a module-finite $A$-algebra such that $\Spec(B) \to
  \Spec(A)$ is a partial weak normalization.
  Set $U \coloneqq \Spec(A) - \{\fm_A\}$.
  Then, the local hull
  \[
    C \coloneqq \Gamma\bigl(U,\tilde{B}\rvert_{U}\bigr)
  \]
  is a module-finite local $B$-algebra with $\depth(C) \ge 2$, such that
  $\Spec(C) \to \Spec(B)$ is a partial weak normalization.
\end{lemma}
\par Here, $\tilde{\cdot}$ denotes the sheaf associated to a module on
an affine scheme.
\begin{proof}
  Throughout this proof, if $R$ is a local ring, then we denote by $\fm_R$ and
  $k_R$ the maximal ideal and residue field of $R$, respectively.
  By replacing $A$ with its reduction $A_\red$, it suffices to consider the case
  when $A$ is reduced.
  \par We first show that $B \to C$ is module-finite, for which it suffices to
  show that $A \to C$ is module-finite.
  By Koll\'ar's finiteness theorem \cite[Thm.\ 2]{Kol17}, $(5) \Rightarrow (1)$,
  it suffices to show
  that for every $\fp \in \Ass_A(B)$, the local hull
  \[
    \Gamma\bigl(U,(A/\fp)^\sim\rvert_U\bigr)
  \]
  is a finitely generated $A$-module.
  Since $A$ and $B$ are reduced and $\Spec(B) \to \Spec(A)$ is a homeomorphism,
  we see that $\Ass_A(A) = \Ass_A(B)$.
  Thus, applying Koll\'ar's finiteness theorem \cite[Thm.\ 2]{Kol17}, $(1)
  \Rightarrow (5)$ to $F = \cO_{\Spec(A)}$, it suffices to show that the local
  hull $\Gamma(U,\tilde{A}\rvert_U)$
  is a finitely generated $A$-module.
  But this is automatic since the natural homomorphism $A \to
  \Gamma(U,\tilde{A}\rvert_U)$ is an isomorphism by \cite[Lem.\
  14]{Kol17}, using the condition $\depth(A) \ge 2$.
  \par We now show that $\Spec(C) \to \Spec(B)$ is a universal homeomorphism.
  Let $A' \coloneqq A^\sh$ be a strict Henselization of $A$, and set $U'
  \coloneqq \Spec(A') - \{\fm_{A'}\}$.
  Denoting $B' \coloneqq B \otimes_A A^\sh$ and $C' \coloneqq C \otimes_A
  A^\sh$, it suffices to show that the morphism
  \begin{equation}\label{eq:specbprimecprime}
    \Spec(C') \longrightarrow \Spec(B')
  \end{equation}
  is a universal homeomorphism by \cite[Cor.\ 2.6.4$(iv)$]{EGAIV2} since $A \to
  A'$ is faithfully flat.
  Setting
  \[
    U' \coloneqq \Spec(A') - \{\fm_{A'}\} \qquad \text{and} \qquad
    V' \coloneqq \Spec(B') - \{\fm_{B'}\},
  \]
  we see that \eqref{eq:specbprimecprime} is an isomorphism over $V'$, and hence
  it suffices to show that there is only one prime ideal in $C'$ lying over
  $\fm_{B'}$, and that the residue field extension $k_{B'} \to k_{C'}$ is
  purely inseparable \cite[Cor.\ 18.12.11]{EGAIV4}.
  First, since $\depth(A') \ge 2$ by \cite[Prop.\ 1.2.16$(a)$]{BH98},
  Hartshorne's connectedness theorem \cite[Thm.\ 5.10.7]{EGAIV2} implies that
  the punctured spectrum $U'$ is connected, and thus, $V'$ is also connected by
  the fact that $\Spec(B') \to \Spec(A')$ is a homeomorphism.
  By the Henselian property of $A' = A^\sh$, the semi-local ring $C'$ is a
  direct product of local rings \cite[Prop.\ 18.5.9$(ii)$]{EGAIV4}.
  But the fact that $V'$ is connected forces there to be only one factor in this
  decomposition, since the morphism
  \eqref{eq:specbprimecprime} is an isomorphism over $V'$, and the support of
  \begin{equation}\label{eq:cprimeiso}
    C' \simeq \Gamma\bigl(U',\tilde{B'}\rvert_{U'}\bigr)
  \end{equation}
  is the closure of the inverse image of $V'$ under \eqref{eq:specbprimecprime}
  by \cite[Def.\ 1, $(3')$]{Kol17}.
  Thus, $C'$ is a local ring, and
  there is therefore only one prime ideal in $C'$ lying over $\fm_{B'}$.
  Note that the isomorphism \eqref{eq:cprimeiso} follows from flat base change
  for local hulls \cite[(13)]{Kol17} since the preimage of $\{\fm_B\}$ in
  $\Spec(B')$ is $\{\fm_{B'}\}$.
  Finally, since $k_{A'}$ is separably closed, the residue field extension
  $k_{B'} \to k_{C'}$ is purely inseparable, and hence
  \eqref{eq:specbprimecprime} is a universal homeomorphism.
  \par We now note that
  \begin{equation}\label{eq:depthofc}
    \depth_{\fm_C}(C) = \depth_{\fm_A}(C) \ge 2,
  \end{equation}
  where the first equality is \cite[Ch.\ 0, Prop.\ 16.4.8]{EGAIV1}, and the
  second inequality follows from \cite[Lem.\ 14]{Kol17} since $C$ is
  module-finite over $A$.
  Moreover, the morphism $\Spec(C) \to \Spec(B)$ is birational since it induces
  an isomorphism between $\Spec(B) - \{\fm_B\}$ and $\Spec(C)
  - \{\fm_C\}$, which are dense in $\Spec(B)$ and $\Spec(C)$,
  respectively.
  It remains to show that $C$ is reduced.
  Since $\Spec(C) \to \Spec(B)$ is birational, we see that $C$ satisfies
  $(R_0)$, and satisfies $(S_1)$ everywhere except possibly at $\fm_C$.
  But \eqref{eq:depthofc} implies
  \[
    \depth_{\fm_C}(C) \ge 2 \ge 1 = \min\bigl\{1,\dim(C)\bigr\}.
  \]
  Thus, $C$ satisfies $(S_1)$ at $\fm_C$, and $C$ is therefore reduced.
\end{proof}
The following result allows us to restrict a partial weak normalization to a
reduced Cartier divisor.
\begin{lemma}[cf.\ {\cite[Prop.\ 4.8]{BF93}}]\label{lem:bfprop48}
  Let $(A,\fm_A)$ be a noetherian local ring, and let $t \in \fm$ be a
  nonzerodivisor such that $A/tA$ is reduced of dimension $\ge1$.
  Let $B$ be a module-finite
  $A$-algebra such that $\Spec(B) \to \Spec(A)$ is a partial weak normalization
  that is an isomorphism along $V(t) - \{\fm_A\}$.
  Then, the local hull
  \[
    C \coloneqq \Gamma\bigl(U,\tilde{B}\rvert_U\bigr)
  \]
  is a module-finite local $B$-algebra such that the composition
  \[
    \Spec(C/tC) \overset{h}{\longrightarrow} \Spec(B/tB)
    \overset{g}{\longrightarrow} \Spec(A/tA)
  \]
  is a partial weak normalization.
\end{lemma}
\begin{proof}
  Throughout this proof, if $R$ is a local ring, then we denote by $\fm_R$
  the maximal ideal of $R$.
  We note that $A$ is reduced since
  reducedness lifts from Cartier divisors
  \cite[Prop.\ 3.4.6]{EGAIV2}.
  \par Set $\bar{A} \coloneqq A/tA$ and $\bar{C} \coloneqq C/tC$.
  We have
  \[
    \depth(\bar{A}) \ge \min\bigl\{1,\dim(\bar{A})\bigr\} \ge 1
  \]
  since $A/tA$ satisfies $(S_1)$ \cite[Prop.\ 5.8.5]{EGAIV2}.
  Thus, we have $\depth(A) \ge 2$ by \cite[Prop.\
  1.2.10$(d)$]{BH98}, and hence $\Spec(C) \to \Spec(B)$ is a partial weak
  normalization by Lemma \ref{lem:bfprop47}.
  Since $g \circ h$ is an integral universal homeomorphism by base change, to
  show that $g \circ h$ is a partial weak normalization, it suffices to show
  that $g \circ h$ is birational and that $\bar{C}$ is reduced.
  \par We start by showing that $g \circ h$ is birational.
  Since $g \circ h$ is a universal homeomorphism and $\bar{A}$ is reduced, to
  show that $g \circ h$ is birational, it suffices to show that $\bar{A} \to
  \bar{C}$ localizes to
  an isomorphism $\bar{A}_{\bar{\fp}} \overset{\sim}{\to}
  \bar{C}_{\bar{\fp}}$ at every minimal
  prime $\bar{\fp} \subseteq \bar{A}$.
  Let $\fp \subseteq A$ be the prime in $A$ corresponding to such a minimal
  prime $\bar{\fp}$.
  Since regularity lifts from Cartier divisors \cite[Ch.\ 0, Cor.\
  17.1.8]{EGAIV1} and $\bar{A}$ is reduced, we see that
  $A_\fp$ is regular, in which case the homomorphism $A \to C$ induces an
  isomorphism $A_\fp \overset{\sim}{\to} C_\fp$ by the fact that regular rings
  are weakly normal; see \eqref{eq:varnorimplications}.
  Thus, $g \circ h$ is birational.
  \par We now show that $\bar{C}$ is reduced.
  By our assumption that $\Spec(B) \to \Spec(A)$ is an isomorphism
  along $V(t) - \{\fm_A\}$, we know that $\Spec(C) \to \Spec(B)
  \to \Spec(A)$ is an isomorphism along $V(t) - \{\fm_A\}$ by
  \cite[Def.\ 1, $(4')$]{Kol17}.
  Since $\bar{A}$ is reduced and $g \circ h$ is birational,
  we see that $\bar{A} \to \bar{C}$ is injective by \cite[Cor.\ 1.2.6]{EGAInew}.
  Thus, the homomorphism $\bar{A} \to \bar{C}$ restricts to an isomorphism
  \[
    \tilde{\bar{A}}\rvert_{\bar{U}} \overset{\sim}{\longrightarrow}
    \tilde{\bar{C}}\rvert_{\bar{U}}
  \]
  over $\bar{U} \coloneqq \Spec(\bar{A}) - \{\fm_{\bar{A}}\}$.
  Since $\bar{A}$ is reduced, this shows that $\bar{C}$ satisfies $(R_0)$, and
  satisfies $(S_1)$ everywhere except possibly at $\fm_{\bar{C}}$.
  It suffices to show $\depth(\bar{C}) \ge 1$, since this would imply
  \[
    \depth(\bar{C}) \ge \min\bigl\{1,\dim(\bar{C})\bigr\} = 1
  \]
  by the assumption that $\dim(\bar{A}) = \dim(\bar{C}) \ge 1$.
  Note that $\depth(C) \ge 2$ by Lemma \ref{lem:bfprop47}, and hence it suffices
  to show that $t$ is a nonzerodivisor on $C$ by \cite[Prop.\
  1.2.10$(d)$]{BH98}.
  Consider the commutative diagram
  \[
    \begin{tikzcd}
      A \rar\dar[hook] & C\dar[hook]\\
      \Frac(A) \rar{\sim} & \Frac(C)
    \end{tikzcd}
  \]
  where the isomorphism of total rings of fractions on the bottom holds by the
  birationality of $g \circ h$ and the two vertical homomorphisms are injective
  since $A$ and $C$ are reduced.
  Since $A \to \Frac(A)$ is flat, it maps $t$ to a nonzerodivisor on $\Frac(A)$.
  By the commutativity of the diagram, $t$ is therefore a nonzerodivisor on $C$.
\end{proof}
We can now show that weak normality satisfies
$(\hyperref[cond:deforms]{\textup{R}_{\textup{IV}}})$.
This statement is due to Bingener and Flenner in the excellent
case \cite{BF93}.
\begin{proposition}[cf.\ {\citeleft\citen{BF93}\citemid Cor.\ 4.1\citeright}]
  \label{prop:weaklynormaldeforms}
  Let $(A,\fm)$ be a noetherian local ring, and let $t \in \fm$ be a
  nonzerodivisor.
  If $A/tA$ is weakly normal, then $A$ is weakly normal.
\end{proposition}
\begin{proof}
  We want to show that for every module-finite $A$-algebra $B$ such that
  $\Spec(B) \to \Spec(A)$ is a partial weak normalization, the homomorphism $A
  \to B$ is an isomorphism.
  We will show that $A_\fp \to B_\fp$ is an isomorphism for every prime ideal
  $\fp \subseteq A$ containing $t$,
  which would suffice since $\fm$ is a prime ideal containing $t$.
  \par Set $\bar{A} \coloneqq A/tA$ and $\bar{B} \coloneqq B/tB$.
  We induce on the height of $\bar{\fp} \coloneqq \fp\bar{A}$.
  If $\height(\bar{\fp}) = 0$, then $\bar{A}_{\bar{\fp}} \simeq A_\fp/tA_\fp$
  is regular since $\bar{A}$ is reduced by \eqref{eq:varnorimplications}, and
  hence $A_\fp$ is regular since regularity lifts from Cartier divisors
  \cite[Ch.\ 0, Cor.\ 17.1.8]{EGAIV1}.
  Thus, $A_\fp$ is weakly normal by \eqref{eq:varnorimplications}, and $A_\fp
  \to B_\fp$ is an isomorphism.
  \par Now suppose that $\height(\bar{\fp}) > 0$.
  By the inductive hypothesis, we know that $A_\fp \to B_\fp$ is an isomorphism
  along $V(t) - \{\fp A_\fp\} \subseteq \Spec(A_\fp)$.
  Applying Lemmas \ref{lem:bfprop47} and \ref{lem:bfprop48} on $A_\fp$, the
  local hull
  \[
    C \coloneqq \Gamma\bigl( \Spec(A_\fp) -
    \{\fp A_\fp\},\tilde{B}_\fp\rvert_{\Spec(A_\fp) -
    \{\fp A_\fp\}}\bigr)
  \]
  is module-finite over $A$ and
  yields a partial weak normalization $\Spec(C) \to \Spec(B)$ such that the
  composition
  \[
    \Spec(C/tC) \longrightarrow \Spec(\bar{B}_{\bar{\fp}}) \longrightarrow
    \Spec(\bar{A}_{\bar{\fp}})
  \]
  is a partial weak normalization.
  Since $\bar{A}_{\bar{\fp}} \simeq A_\fp/tA_\fp$ is weakly normal by
  \cite[Cor.\ IV.2]{Man80}, we see that this composition is an isomorphism.
  The lemma of Nakayama--Azumaya--Krull \cite[Thm.\ 2.2]{Mat89} then implies that $A_\fp \to C$ is an
  isomorphism.
  Since $B_\fp \to C$ is injective by \cite[Cor.\ 1.2.6]{EGAInew}, we see that
  $B_\fp \to C$ is an isomorphism as well, and hence $A_\fp \to B_\fp$ is also
  an isomorphism.
\end{proof}
Since we have seen that both weak normality and seminormality satisfy the
hypotheses in Theorems \ref{thm:grothlocprob} and \ref{thm:grothlocliftprob}
(see Table \ref{table:condslist} and \eqref{eq:varnorimplications}),
we can solve Problems \ref{prob:grothlocprob} and \ref{prob:grothliftprob} for
both weak normality and seminormality.
The proof of Problem \ref{prob:grothlocprob} for seminormality extends a result
of Shimomoto \cite{Shi17} in the finite type case.
\begin{corollary}[cf.\ {\cite[Cor.\ 3.4]{Shi17}}]\label{cor:thmaforwnsn}
  Grothendieck's localization problem \ref{prob:grothlocprob} and the local
  lifting problem \ref{prob:grothliftprob} hold for
  weak normality and seminormality.
\end{corollary}

\subsection{Terminal, canonical, and rational singularities}
We now consider Problem \ref{prob:grothlocprob} for terminal, canonical, and
rational singularities in equal characteristic zero,
using the version of Theorem
\ref{thm:grothlocprob} in \cite[Prop.\ 7.9.8]{EGAIV2}.
\par See \cite[Def.\ 2.8]{Kol13} for the definition of terminal and canonical
singularities, which apply to excellent noetherian schemes with
dualizing complexes.
For rational singularities, we will
work with the following definition:
\begin{citeddef}[{\citeleft\citen{KKMSD73}\citemid p.\ 51\citepunct
  \citen{Mur}\citemid Def.\ 7.2\citeright}]\label{def:ratsings}
  Let $A$ be a quasi-excellent local $\QQ$-algebra.
  We say that $A$ has \textsl{rational singularities} if $A$ is normal
  and if for every resolution of singularities $f\colon W \to \Spec(A)$, we
  have $R^if_*\cO_W = 0$ for all $i > 0$.
  If $X$ is a locally noetherian $\QQ$-scheme whose local rings $\cO_{X,x}$ are
  quasi-excellent,
  we say that $X$ \textsl{has rational singularities} if $\cO_{X,x}$ has
  rational singularities for every $x \in X$.
  By \cite[Lem.\ 7.3]{Mur}, a quasi-excellent local $\QQ$-algebra has
  rational singularities if and only if $\Spec(R)$ has rational
  singularities.
\end{citeddef}
\begin{remark}[{cf.\ \citeleft\citen{KKMSD73}\citemid p.\ 51\citepunct
  \citen{Kol13}\citemid Cor.\ 2.86\citeright}]\label{rem:ratsingsoneres}
  The object $\RR f_*\cO_W$ does not depend on the choice of resolution
  $f\colon W \to \Spec(A)$, since higher direct images of structure sheaves
  vanish for morphisms of noetherian schemes of finite type over a
  quasi-excellent local $\QQ$-algebra (one combines the strategy of \cite[(2) on
  pp.\ 144--145]{Hir64} with the version of elimination of indeterminacies in
  \cite[Ch.\ I, \S3, Main Thm.\ II$(N)$]{Hir64}).
  Thus, to show that a quasi-excellent local $\QQ$-algebra has rational
  singularities, it suffices to show that $A$ is normal and that there exists a
  resolution of singularities $f\colon W \to \Spec(A)$ such that
  $R^if_*\cO_W = 0$ for all $i > 0$.
  \par In particular, this shows that if $X$ is a quasi-excellent $\QQ$-scheme
  of finite Krull dimension,
  then the rational locus is open in $X$:
  Since resolutions of singularities
  exist in this context \cite[Thm.\ 1.1]{Tem08}, the rational locus is
  the normal locus intersected with the complement of $\bigcup_{i>0}
  \Supp(R^if_*\cO_W)$, where $f\colon W \to X$ is a resolution of singularities.
\end{remark}
\par We now show the following strong version of
$(\hyperref[cond:fgfieldext]{\textup{P}_{\textup{IV}}})$:
\begin{lemma}\label{lem:tercanratisgeom}
  Let $R$ be a quasi-excellent local $\QQ$-algebra,
  and let $k \to R$ be a homomorphism from a field of characteristic zero.
  \begin{enumerate}[label=$(\roman*)$]
    \item If $R$ has a dualizing complex, then $\Spec(R)$ has terminal (resp.\
      canonical) singularities if and only if $\Spec(R \otimes_k k')$ has
      terminal (resp.\ canonical) singularities for every finitely generated
      field extension $k \subseteq k'$.
    \item $\Spec(R)$ has rational singularities if and only if $\Spec(R
      \otimes_k k')$ has rational singularities for every finitely generated
      field extension $k \subseteq k'$.
  \end{enumerate}
\end{lemma}
\begin{proof}
  $\Leftarrow$ holds by setting $k = k'$.
  It therefore suffices to show the converse.
  The terminal and canonical cases follow from \cite[Prop.\ 2.15]{Kol13}, since
  $R \to R \otimes_k k'$ is surjective and essentially smooth by base
  change \citeleft\citen{EGAIV2}\citemid Prop.\ 6.8.3$(iii)$\citepunct
  \citen{EGAIV4}\citemid Thm.\ 17.5.1\citeright.
  Here, we note that $R \otimes_k k'$ has a dualizing
  complex by \cite[(2) on p.\ 299]{Har66}.
  For the rational case, we use the criterion for rational singularities stated
  in Remark \ref{rem:ratsingsoneres}.
  Let $f\colon W \to \Spec(R)$ be a resolution of singularities.
  Since $R \to R \otimes_k k'$ is flat with geometrically regular fibers, the
  morphism $f'\colon W \otimes_k k' \to \Spec(R \otimes k')$ is a resolution of
  singularities \cite[Thm.\ 23.7]{Mat89} and satisfies $R^if'_*\cO_{W \otimes_k
  k'} = 0$ for all $i > 0$ by flat base change.
  Finally, $R \otimes_k k'$ is normal by \cite[Cor.\ to Thm.\ 23.9]{Mat89}.
\end{proof}
The main ingredient for terminal and canonical singularities is the following
version of $(\hyperref[cond:deforms]{\textup{R}_{\textup{IV}}})$, which follows
from the deformation results for complex varieties due to Kawamata \cite{Kaw99}
for canonical singularities and Nakayama \cite{Nak04} for terminal
singularities.
Note that similar results do not hold for klt or log canonical singularities
without assuming that the canonical divisor on $A$ is $\QQ$-Cartier;
see \cite[Ex.\ 9.1.7 and Rem.\ 9.1.15]{Ish18}.
\begin{proposition}[cf.\ {\citeleft\citen{Kaw99}\citemid Main Thm.\citepunct
  \citen{Nak04}\citemid Ch.\ VI, Thm.\ 5.2(2)\citeright}]
  \label{prop:tercanratdeform}
  Let $(A,\fm)$ be a local ring essentially of finite type over a field $k$ of
  characteristic zero, and let $t \in \fm$ be a nonzerodivisor.
  If $A/tA$ has terminal (resp.\ canonical) singularities, then $A$
  has terminal (resp.\ canonical) singularities.
\end{proposition}
\begin{proof}
  Note that $A$ is a normal domain by \cite[Prop.\ I.7.4]{Sey72} and
  \cite[Prop.\ 3.4.5]{EGAIV2}.
  \begin{step}\label{step:tercanratdeformft}
    It suffices to show that if $A'$ is a normal domain of finite type over $k$,
    and $t' \in A'$ is a nonzero element such that $A'/t'$ has terminal (resp.\
    canonical) singularities, then $A'$ has terminal (resp.\
    canonical) singularities along $V(t')$.
  \end{step}
  Since $A$ is essentially of finite type over $k$, there exists a ring $A'$ of
  finite type over $k$, a prime ideal $\fp \subseteq A'$ such that $A =
  A'_\fp$, and an element $t' \in A'$ such that $t'$ maps to $t \in A$.
  Since the normal locus in $\Spec(A')$ is open \cite[Cor.\ 6.13.5]{EGAIV2}, we
  can replace $\Spec(A')$ by an affine open subset containing $\fp$ to assume
  that $A'$ is normal domain and that $t'$ is nonzero.
  Since the locus over which $\Spec(A'/t')$ has terminal (resp.\ canonical)
  singularities is open by \cite[Cor.\ 2.12]{Kol13}, we can
  replace $A'$ by an affine open subset again to assume that $A'/t'$ has
  terminal (resp.\ canonical) singularities.
  Then, $A'$ has terminal (resp.\ canonical) singularities by
  assumption, and hence $A'_\fp$ has terminal (resp.\ canonical)
  singularities by \cite[Cors.\ 2.12]{Kol13}.
  \begin{step}
    It suffices to show that the statement in Step \ref{step:tercanratdeformft}
    holds when $k = \CC$.
  \end{step}
  Since $A$ is of finite type over $k$, there exists a subfield $k_0 \subseteq
  k$ that is a finitely generated field extension of $\QQ$ together with a ring
  $A_0$ of finite type over $k_0$ such that $A \simeq A_0 \otimes_{k_0} k$, and
  such that there exists an element $t_0 \in A_0$ mapping to $t \in A$ under the
  homomorphism $A_0 \to A$.
  Note that $A_0$ is a normal domain by \cite[Cor.\ to Thm.\ 23.9]{Mat89} and
  \cite[Prop.\ 2.1.14]{EGAIV2}.
  Since
  $A_0/t_0A_0 \otimes_{k_0} k \simeq A/tA$,
  applying Lemma \ref{lem:tercanratisgeom}, we know that $A_0/t_0A_0
  \otimes_{k_0} \CC$ has terminal (resp.\ canonical) singularities,
  and that moreover, $A$ has terminal (resp.\ canonical)
  singularities if $A_0 \otimes_{k_0} \CC$ does.
  Note that the image of $t_0$ in $A_0 \otimes_{k_0} \CC$ is a nonzerodivisor by
  flat base change, and that $A_0 \otimes_{k_0} \CC$ is normal by
  \cite[Cor.\ to Thm.\ 23.9]{Mat89}.
  Working one direct factor of $A_0 \otimes_{k_0} \CC$ at a time, the
  statement in Step \ref{step:tercanratdeformft} for $k = \CC$ implies that $A_0
  \otimes_{k_0} \CC$ has terminal (resp.\ canonical)
  singularities.
  \begin{step}
    The statement in Step \ref{step:tercanratdeformft} holds for $k = \CC$.
  \end{step}
  The terminal case is \cite[Ch.\ VI, Thm.\ 5.2(2)]{Nak04} and the canonical
  case is \cite[Main Thm.]{Kaw99} (see also
  \cite[Ch.\ VI, Thm.\ 5.2(1)]{Nak04}).
\end{proof}
For rational singularities, we prove
$(\hyperref[cond:deforms]{\textup{R}_{\textup{IV}}})$ in greater generality.
For the proof, we will work with the notion of pseudo-rational
rings defined below, which gives a characteristic-free version of
rational singularities.
\begin{citeddef}[{\cite[\S2]{LT80}}]\label{def:pseudorational}
  Let $(A,\fm)$ be a noetherian local ring of dimension $d$.
  We say that $A$ is \textsl{pseudo-rational} if
  \begin{enumerate}[label=$(\roman*)$,ref=\roman*]
    \item $A$ is normal;
    \item $A$ is Cohen--Macaulay;
    \item The $\fm$-adic completion $\widehat{A}$ of $A$ is reduced; and
    \item\label{def:pseudoratbiratcond}
      For every proper birational morphism $f\colon W \to \Spec(A)$
      with $W$ normal, if $E = f^{-1}(\{\fm\})$ is the closed fiber,
      then the canonical homomorphism
      \[
        H^d_\fm(A) = H^d_{\{\fm\}}(f_*\cO_W)
        \xrightarrow{\delta^d_f(\cO_W)}
        H^d_E(\cO_W)
      \]
      appearing as the edge map in the Leray--Serre spectral sequence for the
      composition of functors $\Gamma_{\{\fm\}} \circ f_* = \Gamma_E$
      is injective.
  \end{enumerate}
\end{citeddef}
We now prove $(\hyperref[cond:deforms]{\textup{R}_{\textup{IV}}})$ for rational
singularities using the strategy in \cite[Prop.\ 3.4]{MS}.
The case when $A$ is essentially of finite type over a field of characteristic
zero is due
to Elkik \cite{Elk78}.
\begin{proposition}[{cf.\ \cite[Thm.\ 5]{Elk78}}]\label{prop:ratdeform}
  Let $(A,\fm)$ be a quasi-excellent local $\QQ$-algebra, and let $t \in \fm$ be
  a nonzerodivisor.
  If $A/tA$ has rational singularities, then $A$ has rational singularities.
\end{proposition}
\begin{proof}
  By \cite[Rem.\ 7.4]{Mur}, a quasi-excellent local $\QQ$-algebra has rational
  singularities if and only if it is pseudo-rational.
  Since $A$ is normal and Cohen--Macaulay by \cite[Prop.\ I.7.4]{Sey72} and
  \cite[Thm.\ 2.1.3$(a)$]{BH98}, and $\widehat{A}$ is reduced by \cite[Cor.\ to
  Thm.\ 23.9]{Mat89}, it suffices to show that for every resolution of
  singularities $f\colon W \to \Spec(A)$, the homomorphism
  $\delta^d_f(\cO_W)$ is injective, where $d = \dim(A)$.
  \par Let $E = f^{-1}(\{\fm\})$.
  Set $W_t = f^{-1}(V(t))$, and let $g\colon W'_t \to W_t$ be a resolution
  of singularities, which exists by \cite[Ch.\ I, \S3, Main Thm.\
  I$(n)$]{Hir64}.
  Consider the commutative diagram
  \[
    \begin{tikzcd}
      0 \rar & H^{d-1}_\fm(A/tA) \rar\dar[hook] & H^d_\fm(A) \rar{t \cdot}
      \dar{\delta^d_f(\cO_W)} & H^d_\fm(A) \rar\dar{\delta^d_f(\cO_W)} & 0\\
      0 \rar & H^{d-1}_E(\cO_{W_t}) \rar\dar & H^d_E(\cO_W) \rar{t \cdot} &
      H^d_E(\cO_W) \rar & 0\\
      & H^{d-1}_{g^{-1}(E)}(\cO_{W'_t})
    \end{tikzcd}
  \]
  where the top half is obtained from \cite[Lem.\ 3.12]{Mur}.
  The top left arrow is injective since the composition in the left column is
  injective by the hypothesis that $A/tA$ has rational singularities, where we
  use the fact that the edge maps in Definition
  \ref{def:pseudorational}$(\ref{def:pseudoratbiratcond})$ behave well under
  composition of morphisms \cite[Prop.\ 1.12]{Smi97}.
  The rows are exact on the left by the fact that $A$ is Cohen--Macaulay and by
  the version of Grauert--Riemenschneider vanishing in \cite[Thm.\
  B*$(i)$]{Mur}, respectively.
  \par Now suppose there exists an element $0 \ne \eta \in
  \ker(\delta^d_f(\cO_W))$.
  Since every element in $H^d_\fm(A)$ is annihilated by a power of $t$, after
  multiplying $\eta$ by a power of $t$ we may assume that $t\eta = 0$, in which
  case $\eta$ lies in the image of $H^{d-1}_\fm(A/tA)$ in the top row.
  The commutativity of the diagram implies that the composition
  $H^{d-1}_\fm(A/tA) \to H^d_\fm(\cO_W)$ is injective.
  Since $\eta \in \ker(\delta^d_f(\cO_W))$ by assumption, this shows that $\eta
  = 0$, which is a contradiction.
\end{proof}
We can now solve Grothendieck's localization problem for terminal, canonical,
and rational singularities using \cite[Prop.\ 7.9.8]{EGAIV2}.
While our Theorems \ref{thm:grothlocprob} and \ref{thm:grothlocprobqe} apply to 
positive or mixed characteristic, we are not able to prove this result in
these contexts in part because we do not know whether Propositions
\ref{prop:tercanratdeform} and \ref{prop:ratdeform} hold in positive or mixed
characteristic.
\begin{corollary}\label{cor:locprobfortercanrat}
  Let $\varphi\colon A \to B$ be a flat local homomorphism of quasi-excellent
  local $\QQ$-algebras,
  where $A$ is quasi-excellent and $B$ is essentially of finite
  type over a field of characteristic zero.
  \begin{enumerate}[label=$(\roman*)$,ref=\roman*]
    \item\label{cor:locprobfortercan}
      Suppose that $B$ is essentially of finite type over a field of
      characteristic zero.
      If the closed fiber of $\varphi$ has terminal (resp.\ canonical)
      singularities, then all fibers of $\varphi$ have terminal (resp.\
      canonical) singularities.
    \item\label{cor:locprobforrat}
      If the closed fiber of $\varphi$ has rational singularities, then all
      fibers of $\varphi$ has rational singularities.
  \end{enumerate}
\end{corollary}
\begin{proof}
  By Lemma \ref{lem:tercanratisgeom}, we do not need to distinguish between
  ``has terminal (resp.\ canonical, rational) singularities'' and
  ``geometrically has terminal (resp.\ canonical, rational) singularities.''
  We want to apply \cite[Prop.\ 7.9.8]{EGAIV2} where the category
  \emph{\textbf{C}} in their notation is the category of schemes
  essentially of finite type over a field of characteristic zero in situation
  $(\ref{cor:locprobfortercan})$, and the category of schemes locally
  esssentially of finite type over a quasi-excellent local $\QQ$-algebra in
  situation $(\ref{cor:locprobforrat})$.
  \begin{itemize}
    \item For \cite[(7.9.7.1)]{EGAIV2},
      the category \emph{\textbf{C}} is stable under morphisms locally
      essentially of finite type by definition.
    \item For \cite[(7.9.7.2)]{EGAIV2}, the terminal and canonical
      locus of a scheme $X$ in \emph{\textbf{C}} is open by \cite[Cor.\
      2.12]{Kol13}, and the rational locus of a scheme $X$ in \emph{\textbf{C}}
      is open by Remark \ref{rem:ratsingsoneres}.
    \item The condition \cite[(7.9.7.3)]{EGAIV2} is our condition
      $(\ref{cond:deforms})$ when $\sC$ is the category of rings whose spectra
      are in \emph{\textbf{C}}, which holds by
      Propositions \ref{prop:tercanratdeform} and \ref{prop:ratdeform}.
    \item The condition \cite[(7.9.7.4)]{EGAIV2} is a version of our condition
      $(\ref{cond:fgfieldext})$ when $\sC$ is the category is the category of
      rings whose spectra are in \emph{\textbf{C}}, which holds by Lemma
      \ref{lem:tercanratisgeom}.
    \item For \cite[(7.9.8.2)]{EGAIV2}, every module-finite
      $A$-algebra has a resolution of singularities by \cite[Ch.\ I,
      \S3, Main Thm.\ I$(n)$]{Hir64}.\qedhere
  \end{itemize}
\end{proof}
We state a global version of this result as well.
A similar statement for the behavior of rational singularities in proper flat
families was shown to us by J\'anos Koll\'ar; see also \cite[Thm.\ 4]{Elk78}.
\begin{corollary}\label{cor:rationalsingglobal}
  Let $f\colon Y \to X$ be a flat morphism of locally noetherian $\QQ$-schemes,
  such that the local rings of $X$ and $Y$ are quasi-excellent.
  \begin{enumerate}[label=$(\roman*)$]
    \item Suppose $f$ maps closed points to closed points.
      If every closed fiber of $f$ has rational singularities, then all
      fibers of $f$ have rational singularities.
      If the local rings of $Y$ are essentially of finite type over
      fields of characteristic zero, and every closed fiber of $f$ has terminal
      (resp.\ canonical) singularities, then all fibers of $f$ have terminal
      (resp.\ canonical) singularities.
    \item Suppose $f$ is closed.
      Then, the locus of points $x \in X$ over which $f^{-1}(x)$ has
      rational singularities is stable under
      generization.
      If the local rings of $Y$ are essentially of finite type over
      fields of characteristic zero, then the locus of points $x \in X$ over
      which $f^{-1}(x)$ has terminal (resp.\ canonical) singularities is stable
      under generization.
  \end{enumerate}
\end{corollary}
\begin{proof}
  This follows from
  Corollary \ref{cor:locprobfortercanrat} using the strategies in Propositions
  \ref{thm:grothlocprobglobal} and \ref{cor:pmorphismgenerizes}.
\end{proof}
For completeness, we end by showing
that pseudo-rationality satisfies $(\ref{cond:descent})$ when
$\sC$ is the category of noetherian rings with residual complexes in the sense
of \cite[Def.\ on p.\ 304]{Har66}.
A special case of this statement was shown to us by Karl Schwede.
Similar results hold for cyclically pure homomorphisms of
$\QQ$-algebras whose local rings are quasi-excellent; see
\citeleft\citen{Bou87}\citemid Thm.\ on p.\
65\citepunct \citen{Mur}\citemid Thm.\ C\citeright.
\begin{proposition}\label{prop:pseudoratdescends}
  Let $\varphi\colon (A,\fm) \to (B,\fn)$ be a flat local homomorphism of
  noetherian local rings, such that $B$ has a residual complex.
  If $B$ is pseudo-rational, then $A$ is pseudo-rational.
\end{proposition}
\begin{proof}
  First, we know that $A$ is normal \cite[Cor.\ to Thm.\ 23.9]{Mat89} and
  Cohen--Macaulay \cite[Thm.\ 2.1.7]{BH98}.
  To prove that $\widehat{A}$ is reduced, we note that
  \[
    \widehat{\varphi}\colon \widehat{A} \longrightarrow \widehat{B}
  \]
  is flat, where $\widehat{B}$ is the $\fn$-adic completion of $B$ \cite[Thm.\
  22.4$(i)$]{Mat89}.
  But $\widehat{B}$ is reduced by assumption, and hence $\widehat{A}$ is reduced
  by \cite[Cor.\ to Thm.\ 23.9]{Mat89}.
  \par It remains to show that condition $(\ref{def:pseudoratbiratcond})$ holds.
  Let $\fq \subseteq B$ be a minimal prime lying over $\fm$.
  The localization $B_\fq$ is pseudo-rational since $B$ has a residual complex
  \cite[\S4, Cor.\ of $(iii)$]{LT80}.
  By replacing $B$ with $B_\fq$, we reduce to the case where $B/\fm B$ is
  zero-dimensional, in which case $\sqrt{\fm B} = \fn$.
  Let $f\colon W \to X \coloneqq \Spec(A)$ be a proper birational morphism,
  where $W$ is normal.
  Set $Y \coloneqq \Spec(B)$, and consider the commutative diagram
  \[
    \begin{tikzcd}
      Z \rar{\nu} & W' \rar{f'}\dar[swap]{g'}
      & Y\dar{g}\\
      & W \rar{f} & X
    \end{tikzcd}
  \]
  where the square is cartesian, where $g \coloneqq \Spec(\varphi)$, and where
  $\nu$ is the normalization of $W' \coloneqq W \times_X Y$.
  Note that $f'$ is proper by base change and birational by flat base change
  \cite[Prop.\ 3.9.9]{EGAInew}, and hence $\nu$ is finite by \cite[$(iv)''$ in
  Rem.\ $(a)$ on pp.\ 102--103]{LT80}.
  We then have the commutative diagram
  \[
    \begin{tikzcd}[column sep=5.5em]
      H^d_{(g' \circ \nu)^{-1}(E)}(\cO_Z)
      & \arrow[hook']{l}[swap]{\delta^d_{f'\circ\nu}(\cO_Z)}
      H^d_{\{\fn\}}\bigl((f'\circ\nu)_*\cO_Z\bigr)
      & \lar[equal] H^d_{\fn}(B)\\
      H^d_{E}\bigl((g' \circ \nu)_*\cO_Z\bigr) \uar
      & \lar[swap]{\delta^d_f((g'\circ\nu)_*\cO_Z)}
      H^d_{\{\fm\}}\bigl((f \circ g' \circ \nu)_*\cO_Z\bigr) \uar
      & \lar[equal] H^d_{\fm}(B) \uar[equal]\\
      H^d_{E}(\cO_W) \uar
      & \lar[swap]{\delta^d_f(\cO_W)}
      H^d_{\{\fm\}}(f_*\cO_W) \uar
      & \lar[equal] H^d_{\fm}(A) \uar[hook]
    \end{tikzcd}
  \]
  where the top half of the diagram is from \cite[Prop.\ 1.12]{Smi97}, and the
  bottom half is from the naturality of $\delta^d_f(-)$ applied to the pullback
  map $\cO_W \to (g' \circ \nu)_*\cO_Z$ \cite[Lem.\
  1.11]{Smi97}.
  The horizontal arrow on the top left is injective by the assumption that $B$
  is pseudo-rational, the vertical arrow on the top right is an equality since
  $\sqrt{\fm B} = \fn$, and
  the vertical arrow on the bottom right is injective by the flatness of
  $\varphi$.
  The commutativity of the diagram shows that $\delta^d_f(\cO_W)$ is injective
  as required.
\end{proof}

\begin{table}[b]
  \begin{ThreePartTable}
    \begin{TableNotes}
      \note{``C--M'' stands for ``Cohen--Macaulay.''\\
      See \cite[Rem.\ 4.2 and pp.\ 1--2]{AF94} for definitions of
      $(CI_n)$, $(G_n)$, $(S_n)$, $\cid$, and $\cmd$.\\
      We list necessary assumptions for the results in the bottom two
      sections of the table.}
    \item[1]\label{tn:thmanormalrn} Either $B$ is universally catenary, or $B
      \otimes_A k$ has geometrically (normal $+$ $(R_n)$) formal fibers.
    \item[2]\label{tn:thmaqea} $A$ is quasi-excellent.
    \item[3]\label{tn:thmafrat} $A$ is quasi-excellent and $B$ is excellent.
    \item[4]\label{tn:thmarn} $\widehat{B}$ is equidimensional.
    \item[5]\label{tn:thmacia} $A$ has complete intersection formal fibers.
    \item[6]\label{tn:thmacma} $A$ has Cohen--Macaulay formal fibers.
    \item[7]\label{tn:thmbrn} $A$ is universally catenary.
    \item[N]\label{tn:thmbnagata} $A/I$ is Nagata.
    \item[$\QQ$]\label{tn:thmarat} $A$ and $B$ are quasi-excellent
      $\QQ$-algebras.
    \item[char0]\label{tn:thmachar0} $A$ is a quasi-excellent
      $\QQ$-algebra and $B$ is essentially of finite type
      over a field of characteristic zero.
    \end{TableNotes}
    {\scriptsize
      \renewcommand{\arraystretch}{1.45}
      \begin{longtable}[c]{ccc@{ $+$ }c}
        \caption{\normalsize Special cases of Problems \ref{prob:grothlocprob}
        and \ref{prob:grothliftprob}.\label{table:resultslist}\\
        All results except those in the bottom section can be
        obtained from our methods.}\\
        \toprule
        $\bR$
        & Grothendieck's localization problem \ref{prob:grothlocprob}
        & \multicolumn{2}{c}{Local lifting problem \ref{prob:grothliftprob}}\\
        \cmidrule(lr){1-1}\cmidrule(lr){2-2}\cmidrule(lr){3-4}
        regular & \cite[Thm.\ on p.\ 297]{And74}
        & \multicolumn{2}{c}{\cite[Thm.\ 3]{Rot79}}\\
        $(R_n) + (S_{n+1})$ & \cite[Prop.\ 1.5]{BI84}
        & \multicolumn{2}{c}{\cite[2.4$(iv)$]{BI84}}\\
        normal & \cite[Prop.\ 2.4]{Nis81}
        & \multicolumn{2}{c}{\cite[Thm.\ on p.\ 154]{Nis81}}\\
        weakly normal & Corollary \ref{cor:thmaforwnsn}
        & \multicolumn{2}{c}{Corollary \ref{cor:thmaforwnsn}}\\
        seminormal & Corollary \ref{cor:thmaforwnsn}
        & \multicolumn{2}{c}{Corollary \ref{cor:thmaforwnsn}}\\
        reduced & \cite[Prop.\ 2.4]{Nis81}
        & \multicolumn{2}{c}{\cite[Prop.\ 3.6]{Mar75}}\\
        complete intersection & \cite[Thm.\ 2]{Tab84}
        & \cite[Thm.\ 2]{Tab84}
        & \cite[Thm.\ 2.3]{BI84}\tnotex{tn:thmbnagata}\\
        Gorenstein & \citeleft\citen{HS78}\citemid Thm.\ 3.3\citepunct
        \citen{Mar84}\citemid Thm.\ 3.2\citeright
        & \cite[Thm.\ 3.2]{Mar84}
        & \cite[Thm.\ 2.3]{BI84}\tnotex{tn:thmbnagata}\\
        Cohen--Macaulay & \cite[Thm.\ 4.1]{AF94}
        & \cite[Thm.\ 4.1]{AF94}
        & \cite[Thm.\ 2.3]{BI84}\tnotex{tn:thmbnagata}\\
        $(CI_n)$ & \citeleft\citen{CI93}\citemid Cor.\ 2.3\citepunct
        \citen{AF94}\citemid Thm.\ 4.5\citeright
        & \multicolumn{2}{c}{\cite[Cor.\ 3.5]{CI93}\tnotex{tn:thmbnagata}}\\
        $(G_n)$ & \citeleft\citen{CI93}\citemid Thm.\ 2.2 and
        Rem.\ 2.8\citepunct \citen{AF94}\citemid Thm.\ 4.5\citeright
        & \cite[Thm.\ 3.2]{Mar84}
        & \cite[Thm.\ 3.4]{CI93}\tnotex{tn:thmbnagata}\\
        $(S_n)$ & \cite[Thm.\ 4.5]{AF94}
        & \cite[Thm.\ 4.1]{AF94}
        & \cite[Thm.\ 3.4]{CI93}\tnotex{tn:thmbnagata}\\
        C--M $+$ $F$-injective & Corollary \ref{cor:applythmacmfi}
        & \multicolumn{2}{c}{Corollary \ref{cor:applythmacmfi}}\\
        \midrule
        normal + $(R_n)$ & \cite[Prop.\ 1.2]{BI84}\tnotex{tn:thmanormalrn}
        & \multicolumn{2}{c}{\cite[2.4$(iii)$]{BI84}}\\
        domain & Corollary \ref{cor:domain}\tnotex{tn:thmaqea}
        & \multicolumn{2}{c}{}\\
        terminal singularities & Corollary
        \ref{cor:locprobfortercanrat}$(\ref{cor:locprobfortercan})$\tnotex{tn:thmachar0}
        & \multicolumn{2}{c}{}\\
        canonical singularities & Corollary
        \ref{cor:locprobfortercanrat}$(\ref{cor:locprobfortercan})$\tnotex{tn:thmachar0}
        & \multicolumn{2}{c}{}\\
        rational singularities & Corollary
        \ref{cor:locprobfortercanrat}$(\ref{cor:locprobforrat})$\tnotex{tn:thmarat}
        & \multicolumn{2}{c}{}\\
        $F$-rational & Corollary \ref{cor:frational}\tnotex{tn:thmafrat} &
        \multicolumn{2}{c}{}\\
        \midrule
        $(R_n)$ & \cite[Thm.\ 1.2]{Ion86}\tnotex{tn:thmarn}
        & \multicolumn{2}{c}{\cite[Thm.\ 1.4]{Ion86}\tnotex{tn:thmbrn} \tnote{,}
        \tnotex{tn:thmbnagata}}\\
        $\cid \le n$ & \cite[Main Thm.\ $(a)$]{AF94}\tnotex{tn:thmacia}
        & \multicolumn{2}{c}{}\\
        $\cmd \le n$ & \cite[Main Thm.\ $(b)$]{AF94}\tnotex{tn:thmacma}
        & \multicolumn{2}{c}{}\\
        \bottomrule
      \insertTableNotes
    \end{longtable}}
  \end{ThreePartTable}
\end{table}
\begin{landscape}
    \centering
    \begin{ThreePartTable}
      \begin{TableNotes}
        \note{``C--M'' stands for ``Cohen--Macaulay.''
        See \cite[Rem.\ 4.2 and pp.\ 1--2]{AF94} for definitions of $(CI_n)$,
        $(G_n)$, $(S_n)$, $\cid$, and $\cmd$.}
        \item[1]\label{tn:rncated} $(\ref{cond:deforms})$ is false when $\bR =$
          ``$(R_n)$'' \cite[Rem.\ 5.12.6]{EGAIV2}, but holds if one restricts
          $\sC$ to be the category of equidimensional catenary
          noetherian rings.
          For ``normal $+$ $(R_n)$,'' we note that normal local rings are
          equidimensional.
        \item[2]\label{tn:domainr1false} $(\ref{cond:ascentphom})$ is false when
          $\bR =$ ``domain''
          \cite[Rems.\ 6.5.5$(ii)$ and 6.15.11$(ii)$]{EGAIV2}.
        \item[3]\label{tn:fratr1} $(\ref{cond:ascentphom})$ holds when $\bR =$
          ``$F$-rational'' if $\sC$ is the category of excellent rings; cf.\
          \cite[Thm.\ 6.4]{Has01}.
          Enescu assumes that $B/\fm B \otimes_{A/\fm} F^e_*(A/\fm)$ is
          noetherian for every $e > 0$, which is used in the proof of
          \cite[Thm.\ 2.19]{Ene00}.
          Enescu states that this latter result
          holds without this assumption as long as ideals generated by systems
          of parameters in the rings $B/\fm B \otimes_{A/\fm} F^e_*(A/\fm)$ are
          Frobenius closed in \cite[Rem.\ 2.20]{Ene00}.
          This condition holds by the proof of \cite[Prop.\ 3.11]{DM}.
        \item[4]\label{tn:cidr1ci} $(\ref{cond:ascentphom})$ holds for
          complete intersection (resp.\ Cohen--Macaulay) homomorphisms when $\bR
          =$ ``$\cid \le n$'' (resp.\ ``$\cmd \le n$'').
          The cited references relate the complete intersection (resp.\
          Cohen--Macaulay) defects of the domain and fiber to that of the
          codomain.
        \item[e]\label{tn:excellent} This property holds if $\sC$ is the category
          of excellent rings.
        \item[$\QQ$]\label{tn:rationalqeqscheme}
          $(\ref{cond:ascentphom})$, $(\ref{cond:deforms})$, and
          $(\ref{cond:generizes})$ hold when $\sC$ is the category of
          $\QQ$-algebras whose local rings are quasi-excellent, and
          $(\ref{cond:openness})$ holds for complete local $\QQ$-algebras.
          For $(\ref{cond:ascentphom})$, Elkik works with rings essentially of
          finite type over a field of characteristic zero, but the same proof
          works after replacing \cite[Thm.\ 2]{Elk78} with Proposition
          \ref{prop:ratdeform}.
        \item[rc]\label{tn:residualcomplex} This property holds for
          pseudo-rationality if $\sC$ is the
          category of noetherian rings that have a residual complex in the sense
          of \cite[Def.\ on p.\ 304]{Har66}.
        \item[CM]\label{tn:cm} This property holds if $\sC$ is the category of
          homomorphic images of Cohen--Macaulay rings.
        \item[char0]\label{tn:eftoverchar0}
          $(\ref{cond:openness})$ holds for complete local $\QQ$-alegbras since
          resolutions of singularities exist \cite[Ch.\ I, \S3, Main Thm.\
          I$(n)$]{Hir64}, and
          the other properties hold if $\sC$ is
          the category of
          rings essentially of finite type over possibly different
          fields of characteristic zero.
          The results for $(\ref{cond:deforms})$ require some extra work; see
          Proposition \ref{prop:tercanratdeform}.
      \end{TableNotes}
      \begingroup
      \tiny
      \renewcommand{\arraystretch}{1.425}
      \begin{longtable}[c]{*{6}{c}}
        \caption{\normalsize Properties satisfying Conditions
        \ref{cond:listof}.\label{table:condslist}\\
        The conditions hold under additional assumptions for
        the properties in the second and third sections of the table.}\\
        \toprule
        $\bR$
        & Ascent $(\ref{cond:ascentphom})$
        & Descent $(\ref{cond:descent})$
        & Openness $(\ref{cond:openness})$
        & Lifting from Cartier divisors $(\ref{cond:deforms})$
        & Localization $(\ref{cond:generizes})$\\
        \cmidrule(lr){1-1} \cmidrule(lr){2-6}
        regular
        & \multicolumn{2}{c}{\cite[Thm.\ 23.7]{Mat89}}
        & \cite[Thm.\ 6.12.7]{EGAIV2}
        & \cite[Ch.\ 0, Cor.\ 17.1.8]{EGAIV1}
        & \cite[Thm.\ 19.3]{Mat89}\\
        $(R_n) + (S_{n+1})$
        & \multicolumn{2}{c}{see $(R_n)$ and $(S_n)$}
        & see $(R_n)$ and $(S_n)$
        & \cite[Lem.\ 0$(ii)$]{BR82}
        & see $(R_n)$ and $(S_n)$\\
        normal
        & \multicolumn{2}{c}{\cite[Cor.\ to Thm.\ 23.9]{Mat89}}
        & \cite[Cor.\ 6.13.5]{EGAIV2}
        & \cite[Prop.\ I.7.4]{Sey72}
        & \cite[Ch.\ V, \S1, n\textsuperscript{o} 5, Prop.\ 16]{BouCA}\\
        weakly normal
        & \cite[Thm.\ 37]{Kol16}
        & \cite[Cor.\ II.2]{Man80}
        & \cite[Thm.\ 7.1.3]{BF93}
        & Proposition \ref{prop:weaklynormaldeforms}
        & \cite[Cor.\ IV.2]{Man80}\\
        seminormal
        & \cite[Thm.\ 37]{Kol16}
        & \cite[Thm.\ 1.6]{GT80}
        & \cite[Cor.\ 2.3]{GT80}
        & \cite[Main Thm.]{Hei08}
        & \cite[Cor.\ 2.2]{GT80}\\
        reduced
        & \multicolumn{2}{c}{\cite[Cor.\ to Thm.\ 23.9]{Mat89}}
        & \cite[Ch.\ II, \S2, n\textsuperscript{o} 6]{BouCA}
        & \cite[Prop.\ 3.4.6]{EGAIV2}
        & \cite[Ch.\ II, \S2, n\textsuperscript{o} 6, Prop.\ 17]{BouCA}\\
        complete intersection
        & \multicolumn{2}{c}{\cite[Thm.\ 2]{Avr75}}
        & \cite[Cor.\ 3.3]{GM78}
        & \cite[Thm.\ 2.3.4$(a)$]{BH98}
        & \cite[Cor.\ 1]{Avr75}\\
        Gorenstein
        & \multicolumn{2}{c}{\cite[Cor.\ 3.3.15]{BH98}}
        & \cite[Cor.\ 1.5]{GM78}
        & \cite[Prop.\ 3.1.19$(b)$]{BH98}
        & \cite[Prop.\ 3.1.19$(a)$]{BH98}\\
        Cohen--Macaulay
        & \multicolumn{2}{c}{\cite[Thm.\ 2.1.7]{BH98}}
        & \cite[Cor.\ 6.11.3]{EGAIV2}
        & \cite[Thm.\ 2.1.3$(a)$]{BH98}
        & \cite[Thm.\ 2.1.3$(b)$]{BH98}\\
        $(CI_n)$
        & \multicolumn{2}{c}{\cite[Prop.\ 1.10]{CI93}}
        & \cite[Prop.\ 1.10]{CI93}
        & \cite[Lem.\ 3.1]{Imb95}
        & \cite[Prop.\ 1.10]{CI93}\\
        $(G_n)$
        & \multicolumn{2}{c}{\citeleft\citen{RF72}\citemid Prop.\ 1\citepunct
        \citen{Pau73}\citemid Prop.\ 1\citeright}
        & \cite[Prop.\ 18$(2)$]{Ooi80}
        & \cite[Prop.\ 3]{RF72}
        & \cite[Cor.\ $(b)$ to Prop.\ 1]{RF72}\\
        $(S_n)$
        & \multicolumn{2}{c}{\cite[Thm.\ 23.9$(iii)$]{Mat89}}
        & \cite[Prop.\ 6.11.2$(ii)$]{EGAIV2}
        & \cite[Lem.\ 0$(i)$]{BR82}
        & \cite[Rem.\ 5.7.3$(iv)$]{EGAIV2}\\
        C--M $+$ $F$-injective
        & \cite[Thm.\ 4.3]{Ene09} & \cite[Lem.\ 4.6]{Has10}
        & \cite[Cor.\ 4.18]{Has10}
        & \cite[Thm.\ 3.4(1)]{Fed83}
        & \cite[Cor.\ 4.11]{Has10}\\
        \midrule
        normal + $(R_n)$
        & \multicolumn{2}{c}{see normal and $(R_n)$}
        & see normal and $(R_n)$
        & see normal and $(R_n)$\tnotex{tn:rncated}
        & see normal and $(R_n)$\\
        domain
        & \textbf{false}\tnotex{tn:domainr1false}
        & \cite[Prop.\ 2.1.14]{EGAIV2}
        & \cite[Lem.\ 1]{Nag59}
        & \cite[Prop.\ 3.4.5]{EGAIV2}
        & \cite[Ch.\ II, \S2, n\textsuperscript{o} 1, Rem.\ 7]{BouCA}\\
        terminal singularities
        &
        &
        & \cite[Cor.\ 2.12]{Kol13}\tnotex{tn:eftoverchar0}
        & \cite[Main Thm.]{Kaw99}\tnotex{tn:eftoverchar0}
        & \cite[Cor.\ 2.12]{Kol13}\tnotex{tn:eftoverchar0}\\
        canonical singularities
        &
        &
        & \cite[Cor.\ 2.12]{Kol13}\tnotex{tn:eftoverchar0}
        & \cite[Ch.\ VI, Thm.\ 5.2(2)]{Nak04}\tnotex{tn:eftoverchar0}
        & \cite[Cor.\ 2.12]{Kol13}\tnotex{tn:eftoverchar0}\\
        rational singularities
        & \cite[Thm.\ 5]{Elk78}\tnotex{tn:rationalqeqscheme}
        & Proposition \ref{prop:pseudoratdescends}\tnotex{tn:residualcomplex}
        & Remark \ref{rem:ratsingsoneres}\tnotex{tn:rationalqeqscheme}
        &
        Proposition \ref{prop:ratdeform}\tnotex{tn:rationalqeqscheme}
        & \cite[Lem.\ 7.3]{Mur}\tnotex{tn:rationalqeqscheme}\\
        $F$-rational
        & \cite[Thm.\ 2.27]{Ene00}\tnotex{tn:fratr1} \tnote{,} \tnotex{tn:excellent}
        & \cite[Prop.\ A.5]{DM}
        & \cite[Thm.\ 3.5]{Vel95}
        & \cite[Thm.\ 4.2$(h)$]{HH94}\tnotex{tn:cm}
        & \cite[Thm.\ 4.2$(f)$]{HH94}\tnotex{tn:cm}\\
        \midrule
        $(R_n)$
        & \multicolumn{2}{c}{\cite[Thm.\ 23.9$(i)$,$(ii)$]{Mat89}}
        & \cite[Prop.\ 6.12.9]{EGAIV2}
        & \cite[Lem.\ 2.1]{Ion86}\tnotex{tn:rncated}
        & \cite[Prop.\ 1.6]{CI93}\\
        $\cid \le n$
        & \multicolumn{2}{c}{\cite[Prop.\ 3.6]{Avr77}\tnotex{tn:cidr1ci}}
        & \cite[Prop.\ 3.4]{Avr77}
        & \cite[Prop.\ 3.10]{Avr77}
        & \cite[Prop.\ 3.8]{Avr77}\\
        $\cmd \le n$
        & \multicolumn{2}{c}{\cite[Cor.\ 6.3.2]{EGAIV2}\tnotex{tn:cidr1ci}}
        & \cite[Prop.\ 6.11.2$(i)$]{EGAIV2}
        & \cite[Ch.\ 0, Prop.\ 16.4.10$(i)$]{EGAIV1}
        & \cite[Prop.\ 6.11.5]{EGAIV2}\\
        \bottomrule
        \insertTableNotes
      \end{longtable}
      \endgroup
    \end{ThreePartTable}
\end{landscape}

\end{document}